\documentclass{amsart}
\usepackage[utf8]{inputenc}
\usepackage{comment}
\usepackage{amscd}
\usepackage{todonotes}
\usepackage{amssymb}
\usepackage{pstricks}
\usepackage{graphicx}
\usepackage{lineno}
\usepackage{subcaption}
\newtheorem{definition}{Definition}
\newtheorem{theorem}{Theorem}

\newtheorem{lemma}{Lemma}
\newtheorem{corollary}{Corollary}
\newtheorem{proposition}{Proposition}
\newtheorem{remark}{Remark}
\newtheorem{example}{Example}
\newtheorem{conjecture}{Conjecture}
\newcommand{\R}{\mathbb{R}}
\numberwithin{equation}{section}

\DeclareMathOperator{\rank}{Range}
\title[]{The dichotomy spectrum approach for a global nonuniform asymptotic stability problem: Triangular case via uniformization}

\author[]{\'Alvaro Casta\~neda}
\author[]{Ignacio Huerta}
\author[]{Gonzalo Robledo}

\address{Universidad de Chile, Departamento de Matem\'aticas. Casilla 653, Santiago, Chile}

\email{castaneda@uchile.cl, ignacio.huerta@uchile.cl,  grobledo@uchile.cl }
\thanks{The first author was funded by FONDECYT Regular Grant 1200653. The second author was funded by FONDECYT POSTDOCTORAL Grant 3210132. The third author was funded by FONDECYT Regular Grant 1210733}

\begin{document}

\maketitle

\begin{abstract}
By considering the nonuniform exponential dichotomy spectrum, we introduce a global asymptotic nonuniform stability conjecture for nonautonomous differential systems, whose restriction to the autonomous case is related to the classical Markus--Yamabe Conjecture: we prove that the conjecture is verified for a family of triangular systems of nonautonomous differential equations satisfying boundedness assumptions. An essential tool to carry out the proof is a necessary and sufficient condition ensuring the property of nonuniform exponential dichotomy for upper block triangular linear differential systems. We also obtain some byproducts having interest on itself, such as, the diagonal significance property in terms on the above mentioned spectrum.  
\end{abstract}

\section{Introduction}

\subsection{State of art} In the last decade, the problem of global stability for ordinary differential equations, also known as the Markus--Yamabe Conjecture, has been revisited from different approaches: i) the case of continuous and discontinuous piecewise autonomous vector fields have been considered by  J. Llibre $\&$ X. Zhang \cite{Llibre1},  L. Menezes \cite{Llibre2}, and Y.Zhang $\&$ X-S. Yang  \cite{ZhangYang}, ii) an infinite--dimensional perspective has been studied by H.M. Rodrigues \textit{et al.} \cite{Rodrigues}, iii) on a nonautonomous context,  D. Cheban \cite{Cheban} worked in the framework of cocycles, while \'A. Casta\~neda and G. Robledo \cite{CR}  established a version of this problem of global stability in terms of the uniform exponential dichotomy spectrum.

Let us recall that the Markus-Yamabe Conjecture is a problem  of global asymptotic stability for continuous autonomous dynamical systems on finite dimension, introduced in 1960 by L. Markus and H. Ya\-ma\-be
\cite{MY}, which states that if the differential system
$\dot{x}=f(x)$,
where $f\colon \mathbb{R}^{n}\to \mathbb{R}^{n}$ of class $C^{1}$, $f(0)=0$ and it
is a Hurwitz vector field, that is, the eigenvalues of the Jacobian matrix of $f$ have negative real part at any $x\in \mathbb{R}^{n}$, or equivalently $Jf(x)$ is a Hurwitz matrix for any $x$, then the origin is globally uniformly asymptotically stable.

It is known that this global stability problem is true
when $ n \leq 2 .$ For details about the proof in the planar case, see  R. Fe{\ss}ler in \cite{F}, A.A. Glutsyuk in \cite{Glu} and C. Guti\'errez in \cite{Gu}. When $ n \geq 3$, the conjecture is false due to the work A. Cima \textit{et al.} \cite{CEGMH}, where it has been founded a polynomial vector field satisfying the hypothesis of the problem, however the corresponding differential system has a solution which escape to infinity.

In spite that the conjecture is now completely resolved on its autonomous classical version, many authors have dedicated to determine vector fields satisfying both the hypothesis and its conclusion (see \cite{CG}, \cite{CGM}, \cite{GC}). Remarkable examples of such setting is the case of triangular and gradient vector fields. Indeed, in a triangular context, L. Markus and H. Yamabe proved that the conjecture is true \cite[Th.4]{MY} while  P. Hartman \cite[p. 539 Corollary 11.2]{Hartman} (see also \cite{Manosas}) showed that the conjecture is also true for gradient vector fields.

The main idea of this article is to settle a global nonuniform stability problem for nonlinear nonautonomous systems
\begin{equation}
\label{sistema2}
\dot{x}=f(t,x).
\end{equation}

In the linear case, it can be proved that the global nonuniform asymptotic stability is consequence of the nonuniform exponential stability, which can be described in terms of a particular nonuniform exponential dichotomy. We emphasize that this property of dichotomy is associated to a spectral theory, which will allow us to emulate the notion of Hurwitz vector fields to the nonuniform framework.

\subsection{The Nonuniform conjecture}

Let us recall that the Global Stability Conjecture is stated in terms of the negativeness of the real part of the eigenvalues of $Jf(x)$
and the attractiveness of the origin has a behavior described by the uniform asymptotic stability. Con\-tra\-ri\-ly,  the stability of a linear nonautonomous system cannot be always determined by its eigenvalues, see \textit{e.g.} \cite[p.310]{MY}. However, we point out 
about the existence of several spectral theories based either on characteristic exponents (Lyapunov, Perron and Bohl exponents) or dichotomies \cite{Dieci}, which allow to describe a wide range of asymptotic stabilities for nonautonomous linear systems, being the uniform asymptotic stability only a particular case. 

In this article, we will work with the property of global nonuniform asymptotic stability to settle a noununiform  Markus--Yamabe conjecture \textbf{(NU--MYC)} and we prove that this conjecture is true for triangular systems.

\subsection{Triangular case setting}
As we above stated, in this article we will prove that this conjecture is verified for a family of triangular vector fields. We point out that the proof is completely different to the ones made in \cite[Th.4]{MY} in an autonomous framework and \cite[Cor.1]{CR} in a uniform nonautonomous case, respectively. In both previous works, as the underlying stability is the uniform one, we can use known results describing the stability of triangular differential systems in terms of its diagonal properties. A tool used in the proof of the triangular case in a nonautonomous uniform context is a result of F. Batelli and  K.J. Palmer in \cite{Batelli} providing necessary and sufficient conditions ensuring that an upper triangular block linear system has a uniform exponential dichotomy on the half line whenever its diagonal subsystems also have this property. 

In order to obtain a similar tool to tackle the nonuniform context, we generalize the result of Batelli and Palmer by using the Lemma of Uniformization introduced by L. Zhou \textit {et al.} in \cite[Lemma 1]{W. Zhang} and its consequences. The generalization has a chain of byproducts, which allow us to conclude that \textbf{NU--MYC} is verified for a family of triangular systems.

\subsection{Structure} The section 2 gives a general setting: i) The subsection 2.1 makes a brief review of qualitative nonuniform properties of linear nonautonomous systems such as bounded growth, contractions, exponential dichotomies and its corresponding spectrum together with their interrelationships, ii) The subsection 2.2 gives an overview the functional framework, namely a set of parametrized norms, needed to introduce the Uniformization Lemma. The section 3 recalls the property of global nonuniform asymptotic stability for an equilibrium of (\ref{sistema2}) and introduces the \textbf{NU--MYC}.

\medskip
 
 The main results are presented in the section 4:

\medskip
 
 \noindent $\bullet$ A technical result --in a nonuniform framework-- 
 compares  an upper triangular linear block system with its corresponding diagonal subsystems, this last ones having the properties of bounded growth and exponential dichotomy. It is proved that the property of dichotomy is preserved for the upper triangular block system
provided that the non diagonal block is bounded in terms a of a parametrized norm.

\medskip
 
\noindent $\bullet$ The above mentioned technical result have several consequences as the characterization of the spectrum of the nonuniform exponential dichotomy for upper triangular block systems and upper triangular systems. In particular, we extend the property of diagonal significance to the nonuniform framework.

\medskip

\noindent $\bullet$ By encompassing the previous results, the \textbf{NU--MYC} is proved for the case of triangular systems.

\subsection{Notations}
Throughout this paper, $|\cdot|$ will denote a vector norm whose induced matrix norm is given by $||\cdot||$. The set $[0,+\infty)$ is denoted by $\mathbb{R}_{0}^{+}$ and the set of square 
$n\times n$ matrices with real coefficients is denoted by $M_{n}(\mathbb{R})$, while $I_{n}$ is the identity matrix. A continuous function $M\colon\mathbb{R}_{0}^{+}\to[1,+\infty)$ will be called a \textit{growth rate}.

\section{Preliminar definitions and contextualization}
\subsection{Nonuniform bounded growth, nonuniform dichotomies and spectrum}
Let us consider the nonautonomous linear differential system
\begin{equation}
\label{lin}
\dot{x}=A(t)x,
\end{equation}
where $A\colon \mathbb{R}_{0}^{+}\mapsto M_{n}(\mathbb{R})$
is a locally integrable matrix function.
A basis of solutions of (\ref{lin}) is denoted by $T(t)$, satisfies the matrix differential equation $\dot{T}(t)=A(t)T(t)$ and its corresponding evolution operator is $T(t,s):=T(t)T^{-1}(s)$. It is straightforward to verify that the solution
of (\ref{lin}) with initial condition $x_{0}$ at $t=t_{0}$ is defined by $x(t,t_{0},x_{0})=T(t,t_{0})x_{0}$.



Similarly as in the uniform case, there exist two definitions of nonuniform bounded growth in the literature:

\begin{definition} The evolution operator $T(t,s)$  of \textnormal{(\ref{lin})}  has a:
\begin{itemize}
\item[a)] Full $(M(s),\nu)$--nonuniform bounded growth if there exist a constant $\nu>0$ and a growth rate $M\colon\mathbb{R}_{0}^{+}\to[1,+\infty)$  such that 
\begin{equation*}
\left \| T (t,s) \right \| \leq M(s)e^{\nu\left | t-s \right|} \quad \textnormal{for any $t,s \in\R_0^{+}$},
\end{equation*}
\item[b)] Half $(M(s),\nu)$--nonuniform bounded growth if there exist a constant $\nu>0$ and a growth rate $M\colon\mathbb{R}_{0}^{+}\to[1,+\infty)$  such that 
\begin{equation*}
\left \| T (t,s) \right \| \leq M(s)e^{\nu (t-s)} \quad \textnormal{for any $t\geq s\geq 0$}.
\end{equation*}
\end{itemize}
\end{definition}

We point out that there are no standard definition of bounded growth in the current literature and we are proposing the previous ones in order distinguished them and its consequences. Note that:
\begin{itemize}
    \item[i)] The property of half $(M(s),\nu)$--nonuniform bounded growth is considered in \cite[p.686]{W. Zhang} under the name of \textit{nonuniform bounded growth}.
    \item[ii)] The property of the full $(M(s),\nu)$--nonuniform bounded growth is treated in the particular case of $M(s)=Me^{\varepsilon s}$ where $\varepsilon\geq 0$ and $M\geq1$ by \cite[p.547]{Chu} and  \cite[p.1892]{Zhang}, also under the name of \textit{nonuniform bounded growth}.
    \item[iii)] We propose to denote the particular case of $M(s)=M\geq 1$ as \textit{uniform bounded growth}, which has been considered respectively by S. Siegmund \cite[p.253]{Siegmund2002} and W. Coppel \cite[pp.8-9]{Coppel} in a full and half version respectively under the name of \textit{bounded growth}. We also highlight the related definitions proposed by K.J. Palmer in \cite[pp.172]{Palmer-06}.
\end{itemize}

From now on, in this paper we will work with the property of half $(Me^{\delta s},\nu)$--nonuniform bounded growth.



The property of exponential dichotomy plays an important role in the study of nonautonomous linear systems. A formal definition adapted to the nonuniform framework is given by:
\begin{definition}
The system (\ref{lin}) has a $(K(s),\gamma)$--nonuniform exponential dichotomy on $\mathbb{R}_{0}^{+}$ if there exist
a family of invariant projections $P(t)\colon \mathbb{R}^{n}\to \mathbb{R}^{n}$ for any
$t\in \mathbb{R}_{0}^{+}$, a constant $\gamma>0$ and a growth rate
$K\colon \mathbb{R}_{0}^{+}\to [1,+\infty)$ such that:
\begin{equation}
\label{Invar}
T(t,s)P(s)=P(t)T(t,s)  \quad \textnormal{for $t, s\geq 0$} ,
\end{equation}
\begin{equation}
\label{DENU-PI}
\left\{\begin{array}{rcl}
||T(t,s)P(s)|| &\leq & K(s)e^{-\gamma(t-s)} \quad \textnormal{for $t\geq s\geq 0$},  \\\\
||T(t,s)[I-P(s)]|| &\leq & K(s)e^{-\gamma(s-t)} \quad \textnormal{for $0\leq t\leq s$}.
\end{array}\right.
\end{equation}
\end{definition}

The property (\ref{Invar}) implies that the range of $P(\cdot)$ is invariant for any $t$, which motivates the name \textit{invariant projectors}, we refer the reader to \cite{Kloeden} for details.

The inequalities (\ref{DENU-PI}) imply that $t\mapsto x(t):=T(t,t_{0})\xi$, the forward solution of (\ref{lin}) passing through $\xi\neq 0$ at $t=t_{0}$, can be splitted by $P(t_{0})$ in $t\mapsto x^{+}(t):=T(t,t_{0})P(t_{0})\xi$ and $t\mapsto x^{-}(t):=T(t,t_{0})[I-P(t_{0})]\xi$, whose behavior for any $t\geq t_{0}$ verifies that
\begin{displaymath}
\begin{array}{rcl}
|T(t,t_{0})P(t_{0})\xi|&\leq& K(t_{0})e^{-\gamma(t-t_{0})}|P(t_{0})\xi| \\\\
(1/K(t_{0}))e^{\gamma(t-t_{0})}|[I-P(t_{0})]\xi|&\leq& |T(t,t_{0})[I-P(t_{0})]\xi|,
\end{array}
\end{displaymath}
that is any solution $x(t)$ of (\ref{lin}) is the sum of 
two solutions $x(t)=x^{-}(t)+x^{+}(t)$ having a dichotomic behavior: the
$(K(s),\gamma)$--\textit{nonuniform exponential contraction} $t\mapsto T(t,t_{0})P(t_{0})\xi$ and the $(K(s),\gamma)$--\textit{nonuniform exponential expansion} $t\mapsto T(t,t_{0})[I-P(t_{0})]\xi$.  In this context, it will be useful to consider the following definition:
\begin{definition} The system (\ref{lin}) is a  $(K(s),\gamma)$--\textit{nonuniform exponential contraction} if it
has a $(K(s),\gamma)$--nonuniform exponential dichotomy on $\mathbb{R}_{0}^{+}$ with the projector $P(t)=I$ for any $t\geq 0$.
\end{definition}

On the other hand, we emphasize that the growth rate $K(\cdot)$ can take a wide range of possible cases: a distinguished one, which is older in the literature, is given by the constant function $K(s):=K\geq 1$ and corresponds to the \textit{uniform exponential dichotomy}
on $\mathbb{R}_{0}^{+}$ since the exponential contractions and expansions 
have an exponential rate which are independent of the initial time $t_{0}$. Having in mind this noteworthy case, we can see that the nonuniform exponential dichotomies have been residually defined as dichotomies where the growth rate $K(\cdot)$ is not a constant function and dependent of $t_0.$

In this article, we will focus in the particular case of nonuniform exponential dichotomy having a growth rate described by $K(s)=Ke^{\varepsilon s}$: 
\begin{definition}
\label{DENU-BV}
The system (\ref{lin}) has a $(Ke^{\varepsilon s} ,\gamma)$--nonuniform exponential dichotomy on $\mathbb{R}_{0}^{+}$ if there exist
a family of invariant projections $P(t)\colon \mathbb{R}^{n}\to \mathbb{R}^{n}$ for any
$t\in \mathbb{R}_{0}^{+}$, a constant $K\geq 1$, and a couple $(\gamma,\varepsilon)$ of constants such that $\varepsilon \in [0,\gamma)$ and
\begin{equation*}
T(t,s)P(s)=P(t)T(t,s)  \quad \textnormal{for $t,s\geq 0$} ,
\end{equation*}

\begin{equation*}
\left\{\begin{array}{rcl}
||T(t,s)P(s)|| &\leq & Ke^{-\gamma(t-s)}e^{\varepsilon s} \quad \textnormal{for $t\geq s\geq 0$},  \\\\
||T(t,s)[I-P(s)]|| &\leq & Ke^{-\gamma(s-t)}e^{\varepsilon s} \quad \textnormal{for $0\leq t\leq s$}.
\end{array}\right.
\end{equation*}
\end{definition}


In \cite{Dragicevic} and references therein 
is considered the case of linear systems (\ref{lin}) having simultaneously the properties of $(Ke^{\varepsilon s},\gamma)$--nonuniform exponential dichotomy and full $(K_{0}e^{\mu s},\nu)$--nonuniform bounded growth. This property is called \textit{nonuniform strong exponential dichotomy} and the next result describes the relation between the properties of (nonuniform) exponential dichotomy and bounded growth:
\begin{lemma}
If the linear system (\ref{lin}) has the properties of $(Ke^{\varepsilon s},\gamma)$--nonuniform exponential dichotomy and full $(K_{0}e^{\mu s},\nu)$--nonuniform bounded growth on $\mathbb{R}_{0}^{+}$, then it follows that the constants satisfy the inequality:
\begin{equation}
\label{constraints-IH}
\nu+\max\{\mu,\varepsilon\}\geq \gamma.
\end{equation}
\end{lemma}

\begin{proof}
The proof will be made by contradiction: we will assume that (\ref{constraints-IH}) is not verified and we have that $\gamma >   \max\{\mu,\varepsilon\} + \nu   $, which implies the inequalities
\begin{equation}
\label{CIH1}
\gamma >  \mu + \nu   \quad \textnormal{and} \quad 
\gamma >   \varepsilon + \nu.  
\end{equation}

Firstly, let us consider the case $t\geq s$: by using properties of matrix norms combined with the invariance property, we can deduce that
\begin{displaymath}
\begin{array}{rcl}
||I|| & = & ||T(t,s)\,T(s,t)||\\\\
& = & ||T(t,s)[P(s)+Q(s)]T(s,t)||\\\\
&\leq &  ||T(t,s)P(s)||\, ||T(s,t)||+||T(t,s)||\,||T(s,t)Q(t)||,
\end{array}
\end{displaymath}
where $Q(s)=I-P(s)$. Now, by using the dichotomy and bounded growth properties, we can deduce that
\begin{displaymath}
\begin{array}{rcl} 
||I|| &\leq& \displaystyle KK_{0}e^{-\gamma(t-s)+\varepsilon s}e^{\nu(t-s)+\mu t}+KK_{0}e^{-\gamma(t-s)+\varepsilon t}e^{\nu(t-s)+\mu s} \\\\
&\leq & \displaystyle KK_{0}\left[e^{(-\gamma+\mu+\nu)t}e^{(\varepsilon+\gamma-\nu)s}+e^{(-\gamma+\varepsilon+\nu)t}e^{(\gamma-\nu+\mu)s}\right].
\end{array}
\end{displaymath}

Notice that (\ref{CIH1}) is equivalent to $-\gamma+\mu+\nu <0$ and $-\gamma+\varepsilon+\nu<0$. Then letting $t\to +\infty$ leads to
$||I||\leq 0$, obtaining a contradiction.

The case $s\geq t$ can be addressed in a similar way. In fact, by following the lines of the previous case, we can easily deduce that
\begin{displaymath}
||I||\leq ||T(t,s)|| \,||T(s,t)P(t)||+||T(t,s)Q(s)||\, ||T(s,t)||    
\end{displaymath}
and by using the dichotomies and bounded growth properties we can deduce that
\begin{displaymath}
\begin{array}{rcl} 
||I|| &\leq& \displaystyle KK_{0}e^{-\gamma(s-t)+\varepsilon t}e^{\nu(s-t)+\mu s}+KK_{0}e^{-\gamma(s-t)+\varepsilon s}e^{\nu(s-t)+\mu t} \\\\
&\leq & \displaystyle KK_{0}\left[e^{(-\nu+\gamma+\varepsilon)t}
e^{(-\gamma+\mu+\nu)t}
+e^{(-\nu+\mu+\gamma)t}e^{(\nu-\gamma+\varepsilon)s}\right],
\end{array}
\end{displaymath}
and a contradiction can be obtained again by letting $s\to +\infty$. 
\end{proof}

\begin{remark}
To the best of our knowledge, there are no results describing the relation between the exponential dichotomy and bounded growth in the nonuniform case. When $\mu=\varepsilon=0$, we recover the property $\gamma\leq \nu$ mentioned by Shi and Xiong \cite[p.823]{Shi} for the uniform framework.
\end{remark}

Moreover, the nonuniform exponential dichotomy has been considered in 
several works as \cite{BV-LNCS,Chu,Zhang} and deserves additional remarks:  

$\bullet$ Note that if $\varepsilon=0$, we recover the uniform exponential dichotomy on $\mathbb{R}_{0}^{+}$ and this prompts to denote the term $e^{\varepsilon s}$ as the \textit{nonuniform part}. 

$\bullet$ In \cite[Theorem 10.22]{BV-LNCS}, by using the Multiplicative Ergodic Theorem \cite[10.27]{BV-LNCS} and the Oseledets-Pesin Reduction result \cite[Theorem 10.28]{BV-LNCS}, it is shown that almost all variational equations obtained from a measure-preserving flow admit nonuniform exponential dichotomy and furthermore, the nonuniformity rate is arbitrarily small.

There exists a spectral theory associated to the $(Ke^{\varepsilon s} ,\gamma)$--nonuniform exponential dichotomy on $\mathbb{R}$, which has been constructed in \cite{Chu,Zhang} and adapted to the half line by \cite{Huerta,Zhu} in the continuous case. The above mentioned spectral theory is based in the following definition:

\begin{definition}\textnormal{(\cite{Chu,Zhang,Zhu})}
The $(Ke^{\varepsilon s},\gamma)$--nonuniform exponential dichotomy spectrum of (\ref{lin}) is the set $\Sigma^{+}(A)$ of $\lambda\in\R$ such that the system 
 \begin{equation}
 \label{sistemaperturbado}
\dot{x}=[A(t)-\lambda I]x
 \end{equation}
does not have a $(Ke^{\varepsilon s},\gamma)$--nonuniform exponential dichotomy on $\R_{0}^{+}.$ The resolvent $\rho(A)$ is defined as $\mathbb{R}\setminus\Sigma^{+}(A)$, namely, the values of $\lambda$ such that the system (\ref{sistemaperturbado}) have a $(Ke^{\varepsilon s},\gamma)$--nonuniform exponential dichotomy on $\mathbb{R}_{0}^{+}$.
\end{definition}


The description of $\Sigma^{+}(A)$ is summarized by the following result:

\begin{proposition}
\label{comp-spec}
Let us consider the evolution operator $T(t,s)$  of (\ref{lin}):
\begin{itemize}
\item[a)] \cite[Th.2,8]{Chu} If $T(t,s)$
has a half $(Me^{\delta s},\nu)$--nonuniform bounded growth, then
its nonuniform spectrum verifies $\Sigma^{+}(A)\subset (-\infty,\nu]$ and is the union of $m$ intervals where $0<m\leq n$, such that:
\begin{displaymath}
\Sigma^{+}(A)=\left\{\begin{array}{c}
     [a_1,b_1]  \\
      \textnormal {or}\\
     (-\infty,b_1]
\end{array}\right\}\cup\;[ a_2,b_2 ]\;\cup\;\cdots\;\cup\;[a_{m-1},b_{m-1}]\;\cup [a_{m},b_{m}].
\end{displaymath}
\item[b)] \textnormal{\cite[Cor.2.11]{Chu},\cite[Th.1.2]{Zhang} } If $T(t,s)$
has a full $(Me^{\delta s},\nu)$--nonuniform bounded growth, then
its nonuniform spectrum verifies $\Sigma^{+}(A)\subset [-\nu,\nu]$ and is the union of $m$ intervals where $0<m\leq n$, such that:
\begin{displaymath}
\Sigma^{+}(A)=
     [a_1,b_1]\cup\;[ a_2,b_2 ]\;\cup\;\cdots\;\cup\;[a_{m-1},b_{m-1}]\;\cup [a_{m},b_{m}],
\end{displaymath}
\end{itemize}
\end{proposition}

For $i\in \{1,\ldots,m\}$,
the intervals $[a_{i},b_{i}]$ are called 
\textit{spectral intervals}  while $\rho_{i+1}(A):=(b_{i},a_{i+1})$ are called \textit{spectral gaps} and there always exists a unbounded spectral gap $\rho_{m+1}(A)=(b_{m},+\infty)$. Notice that for any $\lambda \in \rho_{j}(A)$, by the definition of $\Sigma^{+}(A)$, it follows that the system (\ref{sistemaperturbado}) has a $(K_{\lambda}e^{\varepsilon_{\lambda} s},\gamma_{\lambda})$--nonuniform exponential dichotomy with $P_{j}:=P_j(\cdot)$and it can be proved, see \emph{e.g.} \cite{Chu}, that:
\begin{itemize}
    \item[i)] If the first spectral interval is given by $[a_1,b_1]$, then $P_{1}=0$, $P_{m+1}=I_{n}$ and 
$$
\dim \rank P_{i}<\dim \rank P_{i+1}
\,\, \textnormal{for any $i=1,\ldots,m$}. 
$$
\item[ii)] If the first spectral interval is given by $(-\infty,b_1]$, then $P_{m+1}=I_{n}$ and 
$$
\dim\rank P_{i}<\dim\rank P_{i+1}
\,\, \textnormal{for any $i=2,\ldots,m$}. 
$$
\end{itemize}

The properties of $\Sigma^{+}(A)$ and its spectral gaps provides an alternative characterization of the  $(Ke^{\varepsilon s},\gamma)$--nonuniform contractions.
\begin{lemma}
The system (\ref{lin}) has a $(Ke^{\varepsilon s},\gamma)$--\textit{nonuniform exponential contraction} if and only if
$\Sigma^{+}(A)\subset (-\infty,0)$.
\end{lemma}

\begin{proof}
If $\Sigma^{+}(A)\subset (-\infty,0)$, it follows that $0\in \rho_{m+1}(A)$ and the system
(\ref{sistemaperturbado}) with $\lambda=0$ coincides with (\ref{lin}), which has a nonuniform exponential dichotomy with projector $P_{m+1}=I_{n}$. Then, from Definition \ref{DENU-BV}, it is direct that (\ref{lin}) has a $(Ke^{\varepsilon s},\gamma)$--nonuniform exponential contraction. Now, if the linear system (\ref{lin}) has
a $(Ke^{\varepsilon s},\gamma)$--nonuniform exponential contraction, we know that this equivalent to say that (\ref{lin}) has a $(Ke^{\varepsilon s},\gamma)$--nonuniform exponential dichotomy with the identity as projector,
which has full range, then we have that $\lambda=0\in \rho_{m+1}(A)=(b_{m},+\infty)$, which implies that
$\Sigma^{+}(A)\subset (-\infty,b_{m})\subset (-\infty,0)$ since $b_{m}<0$.
\end{proof}

\subsection{Uniformization Lemma}
Given a linear system (\ref{lin}) having a nonuniform exponential dichotomy and a half nonuniform bounded growth, the Uniformization 
Lemma provides a way to endow it with both an uniform exponential dichotomy and uniform bounded growth. Nevertheless, the price to pay is to work in a functional framework described by parametrized vector and operator norms. This result was developed by L. Zhou, K. Lu and W. Zhang in \cite[p.697]{W. Zhang}. As a previous step to its statement, we
need to consider a family of norms in $\mathbb{R}^{n}$ parametrized by $\mathbb{R}_{0}^{+}$ as $\left\{|\cdot|_{t}\right \}_{t\in\mathbb{R}_{0}^{+}}$, that is, $|\cdot|_{t}$ is a norm of $\mathbb{R}^{n}$ for any $t\in \mathbb{R}_{0}^{+}$. These norms are also known as \textit{Lyapunov norms} and have been introduced by L. Barreira and C. Valls in \cite[Section 5.4.2]{BV-LNCS} by following
an approach inspired in the Lyapunov metric constructed by Y.B. Pesin \cite[Sect.1]{Pesin}.

Let $\{|\cdot|_{t}\}_{t\in \mathbb{R}_{0}^{+}}$ be a family of norms. By following \cite{W. Zhang}, we summarize basic facts about this family:
\begin{itemize}
\item[a)] By equivalence of norms,
there exist two functions $L_{i}\colon \mathbb{R}_{0}^{+}\to (0,+\infty)$ such that
\begin{equation}
\label{ineg01}
L_{1}(t)|x|\leq |x|_{t} \leq L_{2}(t)|x| \quad \textnormal{for any $t\in \mathbb{R}_0^{+}$}.
\end{equation}
\item[b)] The family $\{|\cdot|_{t}\}_{t\in \mathbb{R}_{0}^{+}}$ is continuous if the mapping $t\mapsto |x|_{t}$ is continuous on $\mathbb{R}_{0}^{+}$ for any fixed $x\in \mathbb{R}^{n}$. 
In this case, it follows that
the functions $L_{i}$ from (\ref{ineg01}) are continuous on $\mathbb{R}_{0}^{+}$,
\cite[Prop.1]{W. Zhang}. 
\item[c)] A continuous family $\{|\cdot|_{t}\}_{t\in \mathbb{R}_{0}^{+}}$ is called \textit{uniformly lower bounded} if $t\mapsto L_{1}(t)$ is uniformly bounded by a positive constant $L_{1}>0$ and (\ref{ineg01}) can be replaced by
\begin{equation}
\label{ineg1}
L_{1}|x|\leq |x|_{t} \leq L_{2}(t)|x| \quad \textnormal{for any $t\in \mathbb{R}_0^{+}$}.
\end{equation}
\end{itemize}

\begin{remark}
\label{ejemploCacotada}
Given a continuous family of norms $\{|\cdot|_{\tau}\}_{\tau\in\mathbb{R}_{0}^{+}}$ defined on $\mathbb{R}$,
it can be proved that there exists a continuous function $h\colon \mathbb{R}_{0}^{+}\to (0,\infty)$ such that
\begin{equation}
\label{normaparamdim1}
|x|_{\tau}=h(\tau)|x| \quad \textnormal{for any $x\in \mathbb{R}$}.
\end{equation} 
\end{remark}

In this article, we will assume that  $\left\{|\cdot|_{t}\right\}_{t\in \mathbb{R}_{0}^{+}}$ is a uniformly lower bounded continuous family of norms. In consequence, we are implicitly assuming that $t\mapsto L_{2}(t)$ is a unbounded and continuous
function. The following technical lemma will be useful
\begin{lemma}
Given a couple of linear operators  
$U\colon (\mathbb{R}^{n},|\cdot|_{s})\to (\mathbb{R}^{m},|\cdot|_{t})$
and $U\colon (\mathbb{R}^{n},|\cdot|)\to (\mathbb{R}^{m},|\cdot|)$
with norms defined by
$$
||U||_{s,t}= \sup\limits_{x\neq 0}\frac{|Ux|_{t}}{|x|_{s}} \quad
\textnormal{and} \quad
||U||=\sup\limits_{x\neq 0}\frac{|Ux|}{|x|},
$$
it follows that
\begin{equation}
\label{identidad2}
\frac{1}{\beta(t)}||U||_{s,t}\leq ||U||\leq \beta(s)||U||_{s,t}
\quad
\textnormal{and}
\quad
\frac{1}{\beta(s)}||U||\leq ||U||_{s,t}\leq \beta(t)||U||,
\end{equation}
where $\beta \colon \mathbb{R}_{0}^{+}\to [1,\infty)$ is the continuous and upperly unbounded function:
\begin{equation*}
\beta(\tau)=\frac{L_{2}(\tau)}{L_{1}}.
\end{equation*} 
\end{lemma}

\begin{proof}
By using (\ref{ineg1}) it can be proved that for any $x\in \mathbb{R}^{n}$ it follows that
$$
\frac{|Ux|_{t}}{|x|_{s}}\geq \frac{L_{1}}{L_{2}(s)}\frac{|Ux|}{|x|} 
\quad \textnormal{and} \quad \frac{|Ux|_{t}}{|x|_{s}}\leq \frac{L_{2}(t)}{L_{1}}\frac{|Ux|}{|x|},
$$
and the Lemma follows.
\end{proof}

\begin{remark}
Given a linear operator $U\colon \mathbb{R}^{p}\to \mathbb{R}^{q}$ and
a couple $(t,\tau)$ of positive real numbers. It will be useful to recall the estimations
\begin{equation}
\label{UE1}
|U\xi|_{t}\leq ||U||_{\tau,t}|\xi|_{\tau},
\end{equation}
\begin{equation}
\label{UE2}
|U\xi|_{\tau}\leq ||U||_{\tau,\tau}|\xi|_{\tau},
\end{equation}
\begin{equation}
\label{UE3}
|U\xi|_{\tau}\leq ||U||_{t,\tau}|\xi|_{t}.
\end{equation}
\end{remark}

\begin{lemma}[Uniformization Lemma]
The system \textnormal{(\ref{lin})} has both a $(K(s),\alpha)$ --nonuniform exponential dichotomy and a half $(M(s),\nu)$--nonuniform bounded growth on $\mathbb{R}_0^+$ if and only if there exists a continuous family $\left\{|\cdot|_{t}\right\}_{t \in \mathbb{R}_0^+}$ of norms with a uniform lower bound such that Eq.\textnormal{(\ref{lin})} has a uniform exponential dichotomy with respect to $\left\{|\cdot|_{t}\right\}_{t \in \mathbb{R}_{0}^+}$, i.e., there are a projection $P(t): \mathbb{R}^{n} \rightarrow \mathbb{R}^{n}$ and a couple of constants $\alpha>0$ and $\kappa \geq 1$ such that the invariant decomposition condition (the last two inequalities in \textnormal{(\ref{DENU-PI})}) can be replaced by
$$
\begin{array}{rr}
\left\|T(t, s) P(s)\right\|_{s, t} \leq \kappa e^{-\alpha(t-s)} & \textnormal{ for } t \geq s\geq0  \\
\left\|T(t, s) (I-P(s))\right\|_{s, t} \leq \kappa e^{-\alpha(s-t)} & \textnormal{ for } s \geq t\geq0 
\end{array}
$$
and a half $(\mu,\nu)$--uniform bounded growth with respect to $\left\{|\cdot|_{t}\right\}_{t \in \mathbb{R}_{0}^+ }$, namely,
$$
\|T(t, s)\|_{s, t} \leq \mu e^{\nu(t-s)}, \quad t \geq s\in \mathbb{R}_{0}^{+}.
$$
\end{lemma}

\begin{remark}
\label{IFL}
A meticulous reading of the proof of the Uniformization Lemma \cite[pp.697--700]{W. Zhang} shows that the family of norms $\{|\cdot|_{t}\}$ verifies $|x|_{t}\leq L_{2}(t)|x|$, where $L_{2}(t)= M(t)+K(t)$.
\end{remark}

\begin{remark}
A version of the Uniformization Lemma, where the linear system (\ref{lin}) has the properties of $(Ke^{\varepsilon_{1} s},\gamma)$--nonuniform exponential dichotomy and full $(Me^{\varepsilon_{2} s},\nu)$--nonuniform bounded growth, has been used in \cite{Dragicevic}. Another version of the Uniformization Lemma has been introduced in \cite{Dragicevic2} for nonuniform polynomial dichotomies in evolution equations. 
\end{remark}

\section{nonuniform conjecture}
In order to set the problem of nonuniform nonautonomous stability, we have to recall a formal definition of nonuniform asymptotic stability. For this purpose, let us consider the nonautonomous and nonlinear system of ordinary differential equations:
\begin{equation}
\label{MY1}
\dot{x} = f(t,x)
\end{equation}
where $f\colon \mathbb{R}_{0}^{+}\times\mathbb{R}^{n}\to \mathbb{R}^{n}$ has properties ensuring the
existence, uniqueness and unbounded forward continuation of solutions. The solution of (\ref{MY1}) passing through $x_{0}$ at $t_{0}$ will be denoted
as $t\mapsto x(t,t_{0},x_{0})$. Moreover, we will assume that $f(t,0)=0$ for any
$t\in \mathbb{R}_{0}^{+}$.

In addition, it will be useful to
recall the comparison functions \cite{Khalil,Zhou}:
\begin{itemize}
    \item[$\bullet$] A function $\alpha\colon \mathbb{R}_{0}^{+}\to \mathbb{R}_{0}^{+}$ 
    is a $\mathcal{K}$ function if $\alpha(0)=0$ and it is nondecreasing.
    \item[$\bullet$] A function $\alpha\colon \mathbb{R}_{0}^{+}\to \mathbb{R}_{0}^{+}$ 
    is a $\mathcal{K}_{\infty}$ function if $\alpha(0)=0$, $\alpha(t)\to +\infty$ as $t\to +\infty$ and it is strictly increasing.
    \item[$\bullet$] A function $\alpha\colon \mathbb{R}_{0}^{+}\to (0,+\infty)$ 
    is a $\mathcal{N}$ function if it is nondecreasing.
    \item[$\bullet$] A function $\alpha(t,s)\colon \mathbb{R}_{0}^{+}\times \mathbb{R}_{0}^{+}\to \mathbb{R}_{0}^{+}$ 
    is a $\mathcal{KL}$ function if $\alpha(t,\cdot)\in \mathcal{K}$ and 
    $\alpha(\cdot,s)$ is decreasing with respect to $s$ and $\lim\limits_{s\to +\infty}\alpha(t,s)=0$.
\end{itemize}

\begin{definition}\cite{Zhou}
\label{GASNU}
The origin of (\ref{MY1}) is globally nonuniformly asymptotically stable if, for any $\eta>0$, exists $\delta(t_{0},\eta)>0$ such that
$$
|x_{0}|<\delta\Rightarrow|x(t,t_0,x_0)|<\eta,\quad \forall t\geq t_0\geq0
$$
and for any $x_{0}\in \mathbb{R}^{n}$ it follows that
$\lim\limits_{t\to +\infty}x(t,t_{0},x_{0})=0$, 
or equivalently, there exists 
 $\sigma \in \mathcal{KL}$ and  $\theta \in \mathcal{N}$ such that, for any $x_{0}\in \mathbb{R}^{n}$ it follows that
$$
|x(t,t_{0},x_{0})|\leq \sigma(\theta(t_{0})|x_{0}|,t-t_{0}).
$$
\end{definition}

As pointed out in \cite{Rugh}, the adjective \textit{nonuniform} relies in the fact that
the convergence of $t\mapsto x(t,t_{0},x_{0})$ is described by a function $\theta$ which depends explicitly of
the initial time $t_{0}$, that is, the decay of any solution is dependent $t_{0}$. In spite that this property and its consequences has been observed since the seminal works from Massera \cite{Massera}, its study from a dichotomy spectrum and/or differential Lyapunov inequalities perspectives has relaunched the interest on it.

\begin{remark}
 When (\ref{MY1}) is the linear system (\ref{lin}), it can be proved --with the help of $\mathcal{KL}$ functions-- that if the linear system has a $(Ke^{\varepsilon s},\gamma)$--nonuniform exponential contraction, then the origin is globally nonuniformly asymptotically stable. In fact, we have
 $\sigma(\theta(t_{0})|x_{0}|,t-t_{0})=\theta(t_{0})|x_{0}|e^{-\alpha(t-t_{0})}$ and $\theta(t_{0})=Ke^{\varepsilon t_{0}}$.
\end{remark}

\begin{remark}
The specific case of  $\theta(t_0)=K$ corresponds to the property of global uniform asymptotic stability. Notice that the Definition \ref{GASNU} is usually
known as global asymptotic stability in the literature and we added to adjective ``nonuniform" in order to contextualize with the rest of the work.
\end{remark}

\medskip 

\noindent\textbf{Statement of the conjecture:} As we have set forth the premises 
now we are able to state our nonuniform global stability problem.

\medskip

\begin{conjecture}[\bf{Nonuniform Markus--Yamabe Conjecture (NU--MYC)}]

Let us consider the nonlinear system
\begin{equation}
\label{MY}
\dot{x} = f(t,x)
\end{equation}
where $f\colon \mathbb{R}_{0}^{+}\times\mathbb{R}^{n}\to \mathbb{R}^{n}$. If $f$ satisfies the following conditions
\begin{itemize}
\item[\textbf{(G1)}] $f$ is continuous in $\mathbb{R}_{0}^{+}\times \mathbb{R}^{n}$ and $C^1$ with respect to $x$. Moreover, $f$ is such that the forward solutions are defined in $[t_{0},+\infty)$ for any $t_{0}\geq 0$.
\item[\textbf{(G2)}] $f(t,x)=0$ if $x=0$ for all $t\geq 0$.
\item[\textbf{(G3)}] For any piecewise continuous function  $t\mapsto \omega(t)$, the linear system
\begin{equation*}
\dot{\vartheta} = Jf(t,\omega(t))\vartheta,
\end{equation*}
where $Jf(t,\cdot)$ is the jacobian matrix of $f(t,\cdot)$, 
has a $(Ke^{\varepsilon s},\gamma)$--nonuniform exponential dichotomy spectrum satisfying
\begin{equation*}
\Sigma^{+}
(Jf(t,\omega(t))) \subset (-\infty,0).
\end{equation*}
\end{itemize}

Then the trivial solution of the nonlinear system (\ref{MY}) is globally nonuniformly asymptotically stable.
\end{conjecture}

Let us recall that the autonomous version of the Markus--Yamabe conjecture is stated in terms of the eigenvalues spectrum of the jacobian matrix corresponding to the linearized vector field. In this context, the assumption \textbf{(G3)} mimics the above fact and since both spectra (eigenvalues and nonuniform spectrum) belong to $(-\infty,0)$.

A uniform version of the above conjecture has been recently stated in \cite{CR}, where  the uniform conjecture was verified for scalar systems, triangular systems and a family of quasilinear systems where the nonlinearity has suitable properties.

\begin{remark}\label{RN1}
The \textnormal{\textbf{NU--MYC}} is verified for dimension $n=1$. In fact, the result follows the lines of \cite[Th.1]{CR} replacing the uniform context by the nonuniform one.
\end{remark}

\section{Triangular Vector Fields}

 The aim of this section is to prove \textbf{NU--MYC} for triangular vector fields, that is, to show that the origin is a globally nonuniformly asymptotically stable for (\ref{MY}) when $
f(t,x)$ is triangular vector field.  

As we said previously to Remark \ref{RN1}, in \cite{CR} it was proved that the uniform conjecture is verified for triangular systems and we emphasize that can be seen as a consequence of the scalar case combined with a result of F. Batelli and K.J. Palmer \cite{Batelli} for upper triangular systems whose diagonal subsystems have the uniform exponential dichotomy. In consequence, if we intend to emulate the ideas of the proof for the uniform triangular case, it is necessary to generalize the Batelli--Palmer result to a nonuniform framework.

A first step to cope with this problem is to 
study the relation between the nonuniform exponential dichotomy properties of the upper block triangular systems
\begin{equation}
\label{STB}
\dot{z}=\left[\begin{array}{cc}
A(t) & C(t)\\
0 & B(t)
\end{array}\right]z,
\end{equation}
 with the nonuniform exponential dichotomy properties of the subsystems 
\begin{equation}
\label{subsistemas}
\dot{x}=A(t)x \quad \textnormal{and} \quad \dot{y}=B(t)y,     
\end{equation}
where $x\in \mathbb{R}^{n}$, $y\in \mathbb{R}^{m}$, $z\in \mathbb{R}^{n+m}$, $A\in M_{n}(\mathbb{R})$, $B\in M_{m}(\mathbb{R}$) and $C\in M_{nm}(\mathbb{R})$. As we said, this problem has been addressed by F. Batelli and K.J. Palmer \cite{Batelli} in the context of the uniform exponential dichotomy on $\mathbb{R}_{0}^{+}$ with an extension to the discrete case in \cite{Batelli2}. 

In the case of the  $(K(s),\gamma)$--nonuniform exponential dichotomy on $\mathbb{R}_{0}^{+}$, a first approach was carried out by
L. Tien, L. Niehn and T. Chien in \cite{Tien}. In particular, the first result of \cite{Tien} is:
\begin{proposition}[Theorem 2.1 in \cite{Tien}]
\label{vietnam-0}
 If the system (\ref{STB}) has a $(K(s),\gamma)$--nonuniform exponential dichotomy on $\mathbb{R}_{0}^{+}$ then the decoupled subsystems (\ref{subsistemas}) also have a $(K(s),\gamma)$--nonuniform exponential dichotomy on $\mathbb{R}_{0}^{+}$.
\end{proposition}

As the $(Ke^{\varepsilon s},\gamma)$--nonuniform exponential dichotomy is a particular case of the $(K(s),\gamma)$ --nonuniform exponential dichotomy, the following Corollary is immediate:

\begin{corollary}
\label{vietnam-01}
 If the system (\ref{STB}) has a $(Ke^{\varepsilon s},\gamma)$--nonuniform exponential dichotomy on $\mathbb{R}_{0}^{+}$ then the decoupled subsystems (\ref{subsistemas}) also have a $(Ke^{\varepsilon s},\gamma)$--nonuniform exponential dichotomy  on $\mathbb{R}_{0}^{+}$.
\end{corollary}

Our next result will show the converse of Proposition \ref{vietnam-0},  which follows the lines of the result proved by Batelli and Palmer in \cite[Th.1]{Batelli} but imposes boundedness conditions for $C$ in terms of matrix norms $||\cdot||_{t,t}$ described in the subsection 2.2 and also incorporates a creative use of the Uniformization Lemma.

Let us recall we are assuming that (\ref{lin}) has the property of half $(Me^{\delta s},\nu)$--nonuniform bounded growth and a $(Ke^{\varepsilon s},\gamma)$--nonuniform exponential dichotomy on $\mathbb{R}_{0}^{+}$. Then, by using the Uniformization Lemma and Remark \ref{IFL}, we have that the continuous norms $\{|\cdot|_{t}\}$ satisfy the inequalities 
\begin{equation}
\label{ineg1-BV}
L_{1}|x|\leq |x|_{t} \leq \underbrace{(M+K)}_{:=L_{2}}e^{\theta t}|x| \quad \textnormal{for any $t\in \mathbb{R}_0^{+}$},
\end{equation}
for some $L_{1}>0$, where $\theta:=\max\{\delta,\varepsilon\}$ . By considering these specific norms, the inequalities (\ref{identidad2}) becomes
\begin{equation}
\label{identidad3}
\frac{1}{L}e^{-\theta t}||U||_{s,t}\leq ||U||\leq Le^{\theta s}||U||_{s,t}
\,\,
\textnormal{and}
\,\,\,
\frac{1}{L}e^{-\theta s}||U||\leq ||U||_{s,t}\leq Le^{\theta t}||U||,
\end{equation}
where $L=L_{2}/L_{1}$. Now, the converse result of Corollary \ref{vietnam-01} is given by:
\begin{theorem}
\label{vietnam}
Let us consider the upper block triangular system (\ref{STB}). If
\begin{itemize}
\item[i)] The decoupled systems (\ref{subsistemas}) have the property of half $(M e^{\ell s},\omega)$--nonuniform bounded growth 
and  half $(\tilde{M} e^{\tilde{\ell} s},\tilde{\omega})$--nonuniform bounded growth respectively on $\mathbb{R}_{0}^{+}$,
\item[ii)] The decoupled systems (\ref{subsistemas}) have the properties of $(K e^{\varepsilon s},\alpha)$--nonuniform exponential dichotomy 
and  $(\tilde{K} e^{\tilde{\varepsilon} s},\tilde{\alpha})$--nonuniform exponential dichotomy respectively on $\mathbb{R}_{0}^{+}$,
\item[iii)] The non diagonal block verifies
\begin{equation}
\label{bound-C}
||C||_{\tau,\infty}:=\sup\limits_{\tau\in \mathbb{R}_{0}^{+}}||C(\tau)||_{\tau,\tau}< \infty,
\end{equation}
where 
$$
||C(\tau)||_{\tau,\tau}= \sup\limits_{x\neq 0}\frac{|C(\tau)x|_{\tau}}{|x|_{\tau}} \quad \textnormal{with $|\cdot|_{\tau}$ verifying (\ref{ineg1-BV})},
$$
\end{itemize}
then (\ref{STB}) also has a $(\bar{K}e^{\bar{\varepsilon} s},\bar{\alpha})$-- nonuniform exponential dichotomy $\mathbb{R}_{0}^{+}$, 
where $\bar{K}$ is a constant dependent of $\alpha$ and $\tilde{\alpha}$, $\bar{\varepsilon}\geq0$ and $\bar{\alpha}=\min\{\alpha, \widetilde{\alpha}\}$, and a
half $(\bar{M} e^{\bar{\theta} s},\bar{\omega})$--nonuniform bounded growth. 
\end{theorem}

The proof of this result is plenty of bulky technicalities 
and the Uniformization Lemma combined with the functional setting from \cite{W. Zhang}
play a key role. In order to give continuity to the reading of the article, the proof is written in the Appendix.

\begin{remark}
A result related to Theorem \ref{vietnam}   was formulated  in terms of tempered exponential dichotomies by L. Barreira and C. Valls in \cite[Th.2.3]{BVTempered}. This result not includes the $(Ke^{\varepsilon s},\gamma)$--nonuniform exponential dichotomy considered in Theorem \ref{vietnam}. In fact, our nonuniform term  $D(s)=Ke^{\varepsilon s}$
with $\varepsilon>0$ verifies the asymptotic behavior
$$
\limsup\limits_{s\to \infty}\frac{D(s)}{s}=\varepsilon>0,
$$
which is just the opposite to the property considered in \cite{BVTempered}.
\end{remark}

\begin{remark}
By using the left inequality (\ref{identidad3}), we have that $||C(\tau)||_{\tau,\tau}\leq Le^{\theta \tau}||C(\tau)||$ for any $\tau\geq 0$.
In consequence, a sufficient condition ensuring (\ref{bound-C}) is given by
\begin{equation}
\label{condicionsuficiente}
\sup\limits_{\tau\in\mathbb{R}_{0}^{+}}Le^{\theta \tau}||C(\tau)||<\infty,
\end{equation} 
this means, for example, that if $C(\tau)=e^{-\psi\tau}C_{0}(\tau)$, where $\psi\geq\theta$ and $C_{0}$ is bounded on $\mathbb{R}_{0}^{+}$, then the last inequality is satisfied.
\end{remark}

The following scalar example indicates that the condition (\ref{condicionsuficiente}) is sufficient but not necessary in order to the expression (\ref{bound-C}) will be satisfied.

\begin{example}
Given $\tau\in\mathbb{R}_{0}^{+}$, let us consider $c(\tau):(\mathbb{R},|\cdot|_{\tau})\rightarrow(\mathbb{R},|\cdot|_{\tau})$ such that $\displaystyle\sup_{\tau\in\mathbb{R}_{0}^{+}}|c(\tau)|<\infty$, then by considering Remark \ref{ejemploCacotada} we have that
\begin{equation*}
\begin{array}{rcl}
 \|c(\tau)\|_{\tau,\tau}&=&\displaystyle\sup_{x\neq0}\frac{|c(\tau)x|_{\tau}}{|x|_{\tau}}=\displaystyle\sup_{x\neq0}\frac{h(\tau)|c(\tau)x|}{h(\tau)|x|}=|c(\tau)|
\end{array}
\end{equation*}
and we conclude that $\displaystyle\sup_{\tau\in\mathbb{R}_{0}^{+}}\|c(\tau)\|_{\tau,\tau}=\displaystyle\sup_{\tau\in\mathbb{R}_{0}^{+}}|c(\tau)|<\infty$.
\end{example}

\begin{corollary}
The upper block system
\begin{equation}
\label{STB-2}
\dot{x}=\left[\begin{array}{cccc}
A_{1}(t) & C_{12}(t) & \cdots & C_{1k}(t)\\
0 & A_{2}(t) &\cdots &  C_{2k}(t)\\
\vdots &  &  \ddots & \\
0 & \cdots & \cdots & A_{k}(t) 
\end{array}\right]
x  \quad 
\end{equation}
has a  $(Ke^{\varepsilon s},\alpha)$--nonuniform exponential dichotomy on $\mathbb{R}_{0}^{+}$ if all the diagonal systems
\begin{equation*}
\dot{x}_{i}=A_{i}(t)x_{i}  \quad \textnormal{for $i=1,\ldots,k$}
\end{equation*}
have a half nonuniform bounded growth property 
and a $(K_{i}e^{\varepsilon_{i}s},\alpha_{i})$--nonuniform exponential dichotomy on $\mathbb{R}_{0}^{+}$ (with $i=1,\ldots,k$) provided that the upper diagonal blocks verifies 
\begin{equation}
\label{bound-C+}
\sup\limits_{\tau\in \mathbb{R}_{0}^{+}}||C_{j}(\tau)||_{\tau,\tau}<\infty \quad  \textnormal{for any $j\in  \{2,\ldots,k\}$},
\end{equation}
where $C_{j}(t)$ is defined by 
\begin{displaymath}
C_{j}(t)=\left[\begin{array}{c}
C_{1j}(t)\\
C_{2j}(t)\\
\vdots \\
C_{j-1\,j}(t)
\end{array}\right].
\end{displaymath}
\end{corollary}

\begin{proof}
The proof will be carried out recursively.

\noindent If $k=2$, the result is an immediate consequence of Theorem \ref{vietnam} since the matrix of the system (\ref{STB-2}) has similar structure to (\ref{STB}) with
$A_{1}(t)=A(t)$, $A_{2}(t)=B(t)$ and $C_{12}(t)=C(t)$.

\noindent If $k=3$, the 
system (\ref{STB-2}) can be seen as having a similar structure that (\ref{STB}) with
\begin{displaymath}
A(t)=\left[\begin{array}{cc}
A_{1}(t) & C_{12}(t)\\
0 & A_{2}(t)
\end{array}\right], \quad 
B(t)=A_{3}(t) \quad \textnormal{and} \quad C(t)=C_{3}(t)=\left[\begin{array}{c}
C_{13}(t)\\
C_{23}(t)
\end{array}\right]
\end{displaymath}

Note that the subsystems $\dot{x}=A(t)x$
and $\dot{y}=B(t)y$ have a nonuniform exponential dichotomy
on $\mathbb{R}_{0}^{+}$. The first dichotomy property is a consequence of the case $k=2$ while the second one is an hypothesis. As (\ref{bound-C+}) is verified for $j=3$, Theorem \ref{vietnam} implies that the system (\ref{STB-2}) with $k=3$ has a nonuniform exponential dichotomy on $\mathbb{R}_{0}^{+}$ and the proof is achieved by a recursive way for $k\geq 4$.
\end{proof}

A particular --but important-- byproduct of the above Corollary is the following result when all the diagonal terms $A_{i}(t)$ are scalar functions: 
\begin{corollary}
\label{triangulaire}
The upper triangular system
\begin{equation}
\label{Triang}
\dot{x}=\left[\begin{array}{cccc}
a_{1}(t) & c_{12}(t) & \cdots & c_{1n}(t)\\
0 & a_{2}(t) &\cdots &  c_{2n}(t)\\
\vdots &  &  \ddots & \\
0 & \cdots & \cdots & a_{n}(t) 
\end{array}\right]
x  \quad 
\end{equation}
has a  $(Ke^{\varepsilon s},\alpha)$--nonuniform exponential dichotomy on $\mathbb{R}_{0}^{+}$ if all the scalar differential equations
\begin{equation}
\label{escalares}  
\dot{x}_{i}=a_{i}(t)x_{i}  \quad \textnormal{for $i=1,\ldots,k$}
\end{equation}
have a half ($M_ie^{\theta_{i}s},\omega_{i}$)--nonuniform bounded growth property 
and a $(K_{i}e^{\varepsilon_{i}s},\alpha_{i})$--nonuniform exponential dichotomy on $\mathbb{R}_{0}^{+}$ (with $i=1,\ldots,n$) provided that
\begin{equation}
\label{Bound-C-Tri}
\sup\limits_{\tau\in\mathbb{R}_{0}^{+}}|C_{j}(\tau)|_{\tau,\tau}=\sup\limits_{\tau\in\mathbb{R}_{0}^{+}}\left|\left(\begin{array}{c}
c_{1j}(\tau)\\
c_{2j}(\tau)\\
\vdots \\
c_{j-1\,j}(\tau)
\end{array}\right)\right|_{\tau,\tau}<\infty \quad  \textnormal{for any $j\in  \{2,\ldots,n\}$}.
\end{equation}
\end{corollary}

\begin{remark}
Note that, when $j=2$ in (\ref{Bound-C-Tri}), we have
$$
\sup\limits_{\tau\in\mathbb{R}_{0}^{+}}|C_{2}(\tau)|_{\tau,\tau}=\sup\limits_{\tau\in\mathbb{R}_{0}^{+}}|c_{12}(\tau)|_{\tau,\tau}=\sup\limits_{\tau\in\mathbb{R}_{0}^{+}}\sup\limits_{x\neq0}\frac{|c_{12}(\tau)x|_{\tau}}{|x|_{\tau}}<\infty,
$$
and in this case, the family of norms $\{|\cdot|_{\tau}\}_{\tau\in\mathbb{R}_{0}^{+}}$ satisfies (\ref{normaparamdim1}).
\end{remark}

\begin{remark}
When we consider that $n=2$ in (\ref{Triang}), the scalar systems (\ref{escalares}) admit a $(K_ie^{\varepsilon_{i}s},\alpha_{i})$--nonuniform exponential dichotomy on $\mathbb{R}_{0}^{+}$, with $i=1, 2$  and as in Example \ref{ejemploCacotada}, we assume that $$\sup\limits_{\tau\in\mathbb{R}_{0}^{+}}|c_{12}(\tau)|<\infty,$$
then the Corollary \ref{triangulaire} allow us to ensure that the system  
\begin{equation*}
\dot{x}=\left[\begin{array}{cc}
a_{1}(t) & c_{12}(t)\\
0 & a_{2}(t)
\end{array}\right]
x
\end{equation*}
has a $(Ke^{\varepsilon s},\alpha)$--nonuniform exponential dichotomy on $\mathbb{R}_{0}^{+}$.
\end{remark}

\begin{remark}
We point out that in \cite{Palmer-82}, the author proves an analogous result to Corollary \ref{triangulaire} in a context of uniform exponential dichotomy on the half line. Nevertheless, we emphasize in the difference of our approach.
\end{remark}

\begin{lemma}
\label{dominancia}
Under the assumptions of Corollary \ref{vietnam-01} and Theorem \ref{vietnam}, the $(Ke^{\varepsilon s},\alpha)$--nonuniform exponential dichotomy spectrum of the upper block system (\ref{STB}) verifies
\begin{displaymath}
\Sigma^{+}(A)\cup \Sigma^{+}(B) =\Sigma^{+}(U) \quad \textnormal{with} \quad
U=\left(\begin{array}{cc}
A& C \\
0 & B
\end{array}\right).
\end{displaymath}
\end{lemma}
\begin{proof}
Firstly, note that $ \rho(A)\cap \rho(B)$
is not empty. Indeed, otherwise, we will have that  $[\rho(A)\cap \rho(B)]^{c}=\Sigma^{+}(A)\cup \Sigma^{+}(B)=\mathbb{R}$ and then at least one spectrum is unbounded, obtaining a contradiction with Proposition \ref{comp-spec}.

Secondly, if
$\lambda \in \rho(A)\cap \rho(B)$ we have that the systems
\begin{equation}
\label{subsistemas-b}
\dot{x}=[A(t)-\lambda I]x \quad \textnormal{and} \quad \dot{y}=[B(t)-\lambda I]y     
\end{equation}
have a nonuniform exponential dichotomy and Theorem \ref{vietnam}
implies that $\lambda \in \rho(U)$ and consequently it follows that
$\rho(A)\cap \rho(B)\subset \rho(U)$.

Finally, if $\lambda \in \rho(U)$, the  Corollary \ref{vietnam-01} implies
that the subsystems (\ref{subsistemas-b}) have a nonuniform exponential dichotomy 
on $\mathbb{R}_{0}^{+}$ which is equivalent to $\lambda \in \rho(A)\cap \rho(B)$,
which implies that $\rho(U)\subset \rho(A)\cap \rho(B)$ and the result follows.
\end{proof}

Based on a recursive application of Lemma \ref{dominancia}, we obtain the following descriptions for the nonuniform exponential dichotomy spectrum for an upper block system as in (\ref{STB-2}) and for an upper triangular system as in (\ref{Triang}) respectively.

\begin{corollary}
Under the assumptions of Corollary \ref{dominancia}, the nonuniform exponential dichotomy spectrum of the upper block system (\ref{STB-2}) is described by
$$
\Sigma^{+}(A_{1})\cup \Sigma^{+}(A_{2})\cup \cdots \cup\Sigma^{+}(A_{k}).
$$
\end{corollary}

\begin{corollary}
\label{dominancia3}
Under the assumptions of Corollary \ref{triangulaire}, the nonuniform exponential dichotomy spectrum of the upper triangular system (\ref{Triang}) is described by
$$
\Sigma^{+}(a_{1})\cup \Sigma^{+}(a_{2})\cup \cdots \cup\Sigma^{+}(a_{n}).
$$
\end{corollary}

\begin{remark}
The previous Corollary says that the nonuniform exponential dichotomy spectrum of an upper triangular system coincides with the union of the spectra of the scalar equations (\ref{escalares}). This property is known as diagonal significance and was introduced for the discrete uniform exponential dichotomy spectrum by C. P\"otzsche in \cite{CP}. This fact is immediate in the autonomous case, while - counterintuitively - in the nonautonomous framework is not always verified.   
\end{remark}

Theorem \ref{vietnam} and its consequences provide 
the framework to state and prove our main result, namely, the nonautonomous nonuniform Markus--Yamabe conjecture is verified for triangular system of nonautonomous differential equations whose nondiagonal parts satisfy boundedness conditions described in terms of parametrized norms.

\begin{theorem}
Let us consider the triangular system
\begin{equation}
\label{triangular}
\begin{array}{rcr}
\dot{x}_{1}&=&f_{1}(t,x_{1},x_{2},\ldots,x_{n})\\
\dot{x}_{2}&=&f_{2}(t,x_{2},\ldots,x_{n})\\
 &\vdots&  \\
\dot{x}_{n}&=&f_{n}(t,x_{n}),
\end{array}
\end{equation}  
whose right part, namely $F(t,x)$, verifies \textbf{(G1)} and \textbf{(G2)}. If
for any piecewise continuous function $t\mapsto \theta(t)$ it is verified that
\begin{itemize}
\item[(a)] There exist constants $k_{i}\geq 1$, $\alpha_{i}>0$ and $\varepsilon_{i}\geq0
$, with $\varepsilon_{i}<\alpha_{i}$ such that
$$
\int_{s}^{t}\frac{\partial f_{i}}{\partial x_{i}}(\tau,\theta(\tau))\,d\tau \leq \ln(k_{i})-\alpha_{i}(t-s)+\varepsilon_{i}s
\quad \textnormal{for any $t\geq s\geq 0$ and $i=1,\ldots,n$}.
$$
\item[(b)] For any $j\in \{2,\ldots,n\}$ and any piecewise continuous function $t\mapsto \theta(t)$, the partial derivatives verify
\begin{displaymath}
\displaystyle
\sup\limits_{\tau\in\mathbb{R}_{0}^{+}}\left|\left(\begin{array}{c}
\frac{\partial f_{1}(\tau,\theta(\tau))}{\partial x_{j}}\\\\
\frac{\partial f_{2}(\tau,\theta(\tau))}{\partial x_{j}}\\
\vdots \\
\frac{\partial f_{j-1}(\tau,\theta(\tau))}{\partial x_{j}} 
\end{array}\right)\right|_{\tau,\tau}<\infty.
\end{displaymath}
\end{itemize}
then the trivial solution of (\ref{triangular}) is globally nonuniformly asymptotically stable.
\end{theorem}

\begin{proof}
Firstly, we will prove that \textbf{(G3)} is verified. In fact, the statement (a) implies that, for any $i=1,\ldots,n$, the nonuniform exponential dichotomy spectra of the differential equations 
$$
\dot{x}=\frac{\partial f_i}{\partial x_i}(t,\theta(t))x
$$
verifies
\begin{equation}
\label{sp1}
\Sigma^{+}\left[\frac{\partial f_i}{\partial x_i}(t,\theta(t))\right] \subset (-\infty, 0).
\end{equation}
for any piecewise continuous function $t\mapsto \theta(t)=(\theta_{1}(t),\ldots,\theta_{n}(t))$.

On the other hand, let us recall
that the jacobian matrix $JF$ is upper triangular defined by
\begin{displaymath}
JF(t,x_{1},\ldots,x_{n})_{ij}=\left\{
\begin{array}{ccl}
\frac{\partial f_{i}}{\partial x_{j}}(t,x_{i},\ldots,x_{n}) &\textnormal{if}& i\leq j, \\
0 &\textnormal{if} & i>j.
\end{array}\right.
\end{displaymath}

Then, the statement (b) combined with (\ref{sp1}) and Corollary \ref{dominancia3} imply that
\begin{eqnarray*}
\Sigma^{+}\left [JF(t,\theta(t)) \right ]= \bigcup\limits_{i=1}^{n} \Sigma^{+}\left[\frac{\partial f_i}{\partial x_i}(t,\theta(t))\right]\subset (-\infty,0)
\end{eqnarray*}
and \textbf{(G3)} follows.

Let $t\mapsto (x_1(t),x_2(t),\ldots,x_n(t))$ be a solution of (\ref{triangular}) passing through $(x_{1}^0,x_{2}^0, \ldots x_{n}^0)$ at $t=t_{0}$. Note that the scalar equation
$$
\dot{x}_{n}=f_{n}(t,x_{n}) \quad \textnormal{with} \quad x_{n}(t_{0})=x_{n}^{0}
$$
is a subsystem of (\ref{triangular}) whose solution is denoted by $\phi_{n}(t):=x_{n}(t,t_{0},x_{n}^{0})$ and verifies $\lim\limits_{t\to \infty}\phi_{n}(t)=0$ globally and nonuniformly, as we can see by Remark \ref{RN1}.

Now, we can see that the last two equations of (\ref{triangular}) are 
\begin{displaymath}
\left\{\begin{array}{rcl}
\dot{x}_{n-1}&=&f_{n-1}(t,x_{n-1},x_{n})\\
\dot{x}_{n}&=&f_{n}(t,x_{n})
\end{array}\right.
\end{displaymath}
with initial conditions $(x_{n-1}^{0},x_{n}^{0})$ at $t=t_{0}$. The solution of this system is denoted by
$(\phi_{n-1}(t),\phi_{n}(t))$, where $\phi_{n}$ is defined above and $\phi_{n-1}$ is the solution of the scalar equation
$$
\dot{x}_{n-1}=f_{n}(t,x_{n-1},\phi_{n}(t)) \quad \textnormal{with} \quad x_{n-1}(t_{0})=x_{n-1}^{0}
$$
 and, as before, also verifies $\lim\limits_{t\to \infty}\phi_{n-1}(t)=0$ globally and nonuniformly, as we can see again by Remark \ref{RN1}. The rest of the proof can be achieved in a recursive way.
\end{proof}

\appendix

\section{Proof of Theorem \ref{vietnam}}
\subsection{Preliminaries}
By hypothesis, we know that the linear systems
$\dot{x}=A(t)x$ and $\dot{y}=B(t)y$ have a
$(Ke^{\varepsilon s},\alpha)$ and $(\tilde{K}e^{\tilde{\varepsilon}s},\tilde{\alpha})$--nonuniform
exponential dichotomy on $\mathbb{R}_{0}^{+}$ with projectors
$P^{A}(\cdot)$ and $P^{B}(\cdot)$ respectively: 
\begin{equation}
\label{DD}
\begin{array}{rcl}
||X(t,s)P^{A}(s)|| &\leq & Ke^{-\alpha(t-s)+\varepsilon s} \quad \textnormal{for $t\geq s\geq 0$},\\
||X(t,s)[I_{n}-P^{A}(s)]||&\leq & Ke^{-\alpha(s-t)+\varepsilon s} \quad \textnormal{for $s\geq t\geq 0$},
\end{array}
\end{equation}
and
\begin{equation}
\label{DDD}
\begin{array}{rcl}
||Y(t,s)P^{B}(s)|| &\leq & \tilde{K}e^{-\widetilde{\alpha}(t-s)+\tilde{\varepsilon}s} \quad \textnormal{for $t\geq s\geq 0$},\\
||Y(t,s)[I_{m}-P^{B}(s)]||&\leq & \tilde{K}e^{-\widetilde{\alpha}(s-t)+\tilde{\varepsilon}s} \quad \textnormal{for $s\geq t\geq 0$},
\end{array}
\end{equation}
where $X(t,s)$ and $Y(t,s)$ are its corresponding evolution operators. We also know that the above subsystems have the properties of half $(Me^{ \ell s},\omega)$ and half $(\tilde{M}e^{\tilde{\ell} s},\tilde{\omega})$--nonuniform bounded growth on $\mathbb{R}_{0}^{+}$ respectively, that is
\begin{equation}
\label{HNBG}
||X(t,s)|| \leq  Me^{\omega(t-s)+\ell s} \quad \textnormal{and} \quad
||Y(t,s)||\leq  \tilde{M}e^{\tilde{\omega}(t-s)+\tilde{\ell} s} \quad \textnormal{for $t\geq s\geq 0$},
\end{equation}
then, the Uniformization Lemma can be applied to both subsystems. On one hand, the dichotomy estimations become
$$
\begin{array}{rr}
\left\|X(t, s) P^{A}(s)\right\|_{s, t} \leq \kappa e^{-\alpha(t-s)} & \textnormal{ for } t \geq s\geq0,  \\
\left\|X(t, s) [I-P^{A}(s)]\right\|_{s, t} \leq \kappa e^{-\alpha(s-t)} & \textnormal{ for } s \geq t\geq0, \end{array}
$$
and
$$
\begin{array}{rr}
\left\|Y(t, s) P^{B}(s)\right\|_{s, t} \leq \tilde{\kappa} e^{-\tilde{\alpha}(t-s)} & \textnormal{ for } t \geq s\geq0,  \\
\left\|Y(t, s) [I-P^{B}(s)]\right\|_{s, t} \leq \tilde{\kappa} e^{-\tilde{\alpha}(s-t)} & \textnormal{ for } s \geq t\geq 0,
\end{array}
$$
while the half bounded growth properties become
\begin{equation}
\label{LUPBG}
\|X(t, s)\|_{s, t} \leq \mu e^{\omega(t-s)} \quad \textnormal{and} \quad \|Y(t, s)\|_{s, t} \leq \tilde{\mu} e^{\tilde{\omega}(t-s)}\quad t \geq s\geq 0.
\end{equation}

\bigskip

The Uniformization Lemma also ensures the existence of two family of norms: a family $\{|\cdot|_{t}^{A}\}_{t}$ in $\mathbb{R}^{n}$ and $\{|\cdot|_{t}^{B}\}_{t}$ in $\mathbb{R}^{m}$. The inequalities (\ref{ineg1}) are verified with $L_{1}^{A}$,$L_{1}^{B}$,$L_{1}^{A}(t)$ and $L_{2}^{B}(t)$ respectively. By Remark \ref{IFL}, we know that
\begin{displaymath}
L_{2}^{A}(t)=Me^{\ell t}+Ke^{\varepsilon t} \quad \textnormal{and} \quad  L_{2}^{B}(t)=\tilde{M}e^{\tilde{\ell} t}+\tilde{K}e^{\tilde{\varepsilon} t}
\end{displaymath}
and the inequalities (\ref{identidad2}) are verified with 
\begin{equation}
\label{cotas-beta}
\beta^{A}(t)=\frac{L_{2}^{A}(t)}{L_{1}^{A}}\leq Le^{\theta t}   \quad \textnormal{and} \quad \beta^{B}(t)=\frac{L_{2}^{B}(t)}{L_{1}^{B}}\leq Le^{\theta t}, 
\end{equation}
where 
\begin{equation}
\label{cota-L}
L=\max\left\{\frac{M+K}{L_{1}^{A}},\frac{\tilde{M}+\tilde{K}}{L_{1}^{B}}\right\} \quad \textnormal{and} \quad \theta=\max\{\ell,\tilde{\ell},\varepsilon,\tilde{\varepsilon}\}.
\end{equation}

These constants $L$ and $\theta$ are useful to state the following result:
\begin{lemma}
\label{LTNU}
For any $t,\tau,s\geq 0$, the evolution operators $X$ and $Y$ verify
\begin{displaymath}
||X(t,\tau)Z(\tau)C(\tau)V(\tau)Y(\tau,s)||\leq Le^{\theta s}||C||_{\tau,\infty}
||X(t,\tau)Z(\tau)||_{\tau,t}||V(\tau)Y(\tau,s)||_{s,\tau},
\end{displaymath}
where $Z(\tau)$ is either $P^{A}(\tau)$ or $I_{n}-P^{A}(\tau)$ and $V(\tau)$ is either $P^{B}(\tau)$ or $I_{m}-P^{B}(\tau)$.
\end{lemma}
\begin{proof}
Let $\xi\in \mathbb{R}^{m}\setminus\{0\}$. By using (\ref{ineg1-BV}) followed by (\ref{UE1}) and recalling the dimensions of $C(\cdot)$, we have that
\begin{displaymath}
\begin{array}{rcl}
|X(t,\tau)Z(\tau)C(\tau)V(\tau)Y(\tau,s)\xi|&\leq& \displaystyle
\frac{1}{L_1^{A}}|X(t,\tau)Z(\tau)C(\tau)V(\tau)Y(\tau,s)\xi|_{t}\\\\
&\leq &\displaystyle \frac{1}{L_1^{A}}||X(t,\tau)Z(\tau)||_{\tau,t}|C(\tau)V(\tau)Y(\tau,s)\xi|_{\tau}.
\end{array}
\end{displaymath}

By using (\ref{UE2}) followed by (\ref{UE3}), (\ref{bound-C}), (\ref{ineg1-BV}) combined with (\ref{cotas-beta})--(\ref{cota-L}), it follows that
\begin{displaymath}
\begin{array}{rcl}
|X(t,\tau)Z(\tau)C(\tau)V(\tau)Y(\tau,s)\xi| &\leq &\displaystyle \frac{1}{L_1^{A}}||X(t,\tau)Z(\tau)||_{\tau,t}||C(\tau)||_{\tau,\tau}|V(\tau)Y(\tau,s)\xi|_{\tau} \\\\
&\leq &\displaystyle \frac{1}{L_1^{A}}||X(t,\tau)Z(\tau)||_{\tau,t}||C(\tau)||_{\tau,\tau}||V(\tau)Y(\tau,s)||_{s,\tau}|\xi|_{s}\\\\
&\leq& Le^{\theta s}||C||_{\tau,\infty}||X(t,\tau)Z(\tau)||_{\tau,t}\,||V(\tau)Y(\tau,s)||_{s,\tau}|\xi|
\end{array}
\end{displaymath}
and the Lemma follows.
\end{proof}

We will follow the lines of the work carried out by F. Batelli and K.J. Palmer in \cite{Batelli}, which proved that evolution operator of the triangular system (\ref{STB}) is given by
\begin{displaymath}
T(t,s)=\left[\begin{array}{cc}
X(t,s) & W(t,s) \\
0 & Y(t,s)
\end{array}\right]
\end{displaymath}
where $W$ is a $n\times m$ matrix defined by
\begin{equation}
\label{W}
W(t,s):=\int_{s}^{t}X(t,\tau)C(\tau)Y(\tau,s)\,d\tau.
\end{equation}

In addition, as in \cite{Batelli}, let us consider: 
\begin{displaymath}
P(t)=T(t,0)\left[\begin{array}{cc}
P^{A}(0) & \mathcal{L}P^{B}(0) \\
0 & P^{B}(0)
\end{array}\right]T(0,t)=\left[\begin{array}{cc}
P^{A}(t) & R(t) \\
0 & P^{B}(t)
\end{array}\right],
\end{displaymath}
where $\mathcal{L}\colon \mathcal{R}P^{B}(0)\to (\mathcal{R}P^{A}(0))^{\perp}$ is the \textit{linking operator} defined by 
\begin{displaymath}
\mathcal{L}\eta=-\int_{0}^{\infty}[I_{m}-P^{A}(0)]X(0,\tau)C(\tau)Y(\tau,0)\eta\,d\tau,
\end{displaymath}
which plays an important role in the proof of Proposition \ref{vietnam-0}. Moreover, $R(t)$ satisfies the matrix differential equation
$$
\dot{R}=A(t)R-RB(t)+C(t)P^{B}(t)-P^{A}(t)C(t) \quad \textnormal{with  $R(0)=\mathcal{L}P^{B}(0)$},
$$
whose solution is defined by
\begin{equation}
\label{def-R}
\begin{array}{rcl}
R(t) &=& X(t,0)R(0)Y(0,t)\\\\
&&\displaystyle +\int_{0}^{t}X(t,\tau)[C(\tau)P^{B}(\tau)-P^{A}(\tau)C(\tau)]Y(\tau,t)\,d\tau,
\end{array}
\end{equation}
or alternatively as follows
\begin{displaymath}
\begin{array}{rcl}
R(t)
&=&
\displaystyle \underbrace{-\int_{t}^{+\infty}X(t,\tau)[I_{n}-P^{A}(\tau)]C(\tau)P^{B}(\tau)Y(\tau,t)\,d\tau}_{=R_{1}(t)}\\\\
&&\displaystyle \underbrace{-\int_{0}^{t}X(t,\tau)P^{A}(\tau)C(\tau)[I_{m}-P^{B}(\tau)]Y(\tau,t)\,d\tau}_{=R_{2}(t)}.
\end{array}
\end{displaymath}

The Lemma \ref{LTNU} will be helpful to provide an estimation for $R(t)$:
\begin{lemma}
For any $t\geq 0$, the matrix $R(t)$ verifies
\begin{equation}
\label{Bound-R}
||R(t)||\leq \frac{2\kappa\widetilde{\kappa}}{\widetilde{\alpha}+\alpha}Le^{\theta t}||C||_{\tau,\infty}.
\end{equation}
\end{lemma}
\begin{proof}
Let $\xi \in \mathbb{R}^{m}\setminus \{0\}$.
By using Lemma \ref{LTNU} we can deduce
\begin{displaymath}
\begin{array}{rcl}
|R_{1}(t)\xi|&\leq & \displaystyle Le^{\theta t}||C||_{\tau,\infty}\int_{t}^{\infty}||X(t,\tau)[I_{n}-P^{A}(\tau)]||_{\tau,t} \left \|P^{B}(\tau)Y(\tau,t)\right \|_{t,\tau}|\xi|\,d\tau.
\end{array}
\end{displaymath}

Now, we apply the Uniformization Lemma to
(\ref{DD}) and (\ref{DDD}) respectively, and
we can deduce that
\begin{displaymath}
\begin{array}{rcl}
|R_{1}(t)\xi|&\leq & \displaystyle
Le^{\theta t}||C||_{\tau,\infty}\int_{t}^{\infty}\left \|X(t,\tau)[I_{n}-P^{A}(\tau)]\right \|_{\tau,t}  ||P^{B}(\tau)Y(\tau,t)||_{t,\tau}|\xi|\,d\tau\\\\
&\leq& \displaystyle \kappa\,\widetilde{\kappa}Le^{\theta t}||C||_{\tau,\infty}|\xi|\,e^{(\alpha+\widetilde{\alpha})t}\int_{t}^{\infty}e^{-(\alpha+\widetilde{\alpha})\tau}\,d\tau\\\\
&\leq& \displaystyle \frac{\kappa\, \widetilde{\kappa}}{\alpha+\widetilde{\alpha}}Le^{\theta t}||C||_{\tau,\infty}|\xi|.
\end{array}
\end{displaymath}

Similarly, for the second term, by using Lemma \ref{LTNU} followed by the Uniformization Lemma, we can deduce that
$$
\begin{array}{rcl}
|R_{2}(t)\xi|&\leq& \displaystyle
Le^{\theta t}||C||_{\tau,\infty}\int_{0}^{t}
||X(t,\tau)P^{A}(\tau)||_{\tau,t}\,||[I_{m}-P^{B}(\tau)]Y(\tau,t)||_{t,\tau}|\xi|\,d\tau \\\\
&\leq& \displaystyle \kappa\,\widetilde{\kappa}Le^{\theta t}||C||_{\tau,\infty}|\xi|\,e^{-(\alpha+\widetilde{\alpha})t}\int_{0}^{t}e^{(\alpha+\widetilde{\alpha})\tau}\,d\tau\\\\
&\leq &\displaystyle \frac{\kappa \widetilde{\kappa}}{\alpha+\widetilde{\alpha}}Le^{\theta t}||C||_{\tau,\infty}|\xi|,
\end{array}
$$
and the inequality (\ref{Bound-R}) follows.
\end{proof}

We will verify that $P(\cdot)$ is an invariant projector. In fact, the property
$P^{2}(t)=P(t)$ for any $t\geq 0$ is a consequence of its own definition, while the next result proves its invariance. This last property has not been proved in \cite{Batelli} and, in spite that can be deduced easily, we will prove it. 

\begin{lemma}
\label{T-INV}
The projector $P(\cdot)$ is invariant, namely, it verifies the property
\begin{displaymath}
T(t,s)P(s)=P(t)T(t,s)  \quad \textnormal{for any $t,s\geq 0$}.
\end{displaymath}
\end{lemma}
\begin{proof}
Notice that
\begin{displaymath}
T(t,s)P(s)=\left[\begin{array}{cc}
X(t,s)P^{A}(s) & W(t,s)P^{B}(s)+X(t,s)R(s) \\
0 & Y(t,s)P^{B}(s)
\end{array}\right].
\end{displaymath}
As $P^{A}(\cdot)$ and $P^{B}(\cdot)$ are invariant projectors, we can see that the Lemma follows if and only if
\begin{equation*}
R(t)Y(t,s)+P^{A}(t)W(t,s)=W(t,s)P^{B}(s)+X(t,s)R(s) \quad \textnormal{for any $t,s\geq 0$.}
\end{equation*}

By defining $R_{t,s}(0):=X(t,0)R(0)Y(0,s)$, using (\ref{W}) and (\ref{def-R})
and considering $t\geq s$, we can easily deduce that
\begin{displaymath}
\begin{array}{rcl}
R(t)Y(t,s)+P^{A}(t)W(t,s)&=&\displaystyle R_{t,s}(0)+\int_{0}^{t}X(t,\tau)[C(\tau)P^{B}(\tau)-P^{A}(\tau)C(\tau)]Y(\tau,s)\,d\tau \\\\
&&\displaystyle+
\int_{s}^{t}X(t,\tau)P^{A}(\tau)C(\tau)Y(\tau,s)\,d\tau \\\\
&=& \displaystyle R_{t,s}(0)+\int_{0}^{s}X(t,\tau)[C(\tau)P^{B}(\tau)-P^{A}(\tau)C(\tau)]Y(\tau,s)\,d\tau\\\\
&&\displaystyle +\int_{s}^{t}X(t,\tau)C(\tau)P^{B}(\tau)Y(\tau,s)\,d\tau \\\\
&=& X(t,s)R(s)+W(t,s)P^{B}(s).
\end{array}
\end{displaymath}

A similar identity can be deduced considering 
$t<s$ and the Lemma follows.
\end{proof}

Gathering the above results, it can be proved that the triangular system (\ref{STB}) has a ($\bar{K}e^{\bar{\varepsilon}s},\bar{\alpha}$)--nonuniform
exponential dichotomy on $\mathbb{R}_{0}^{+}$ with the above defined invariant projector $P(t)$.

\begin{lemma}
\label{Lema-W}
There exist a constant $K_{3}\geq1$, $\alpha_{3}>0$ and $\varepsilon_{3}\geq0$, where $\varepsilon_3<\alpha_3$ such that
\begin{equation}
\label{DENUa1}
||W(t,s)P^{B}(s)+X(t,s)R(s)||\leq K_{3}e^{-\alpha_{3}(t-s)+\varepsilon_3 s} \quad \textnormal{for 
$t\geq s\geq 0$}.
\end{equation}
\end{lemma}

\begin{proof}
In order to deduce this estimation, we will write
\begin{displaymath}
\begin{array}{rcl}
W(t,s)P^{B}(s)+X(t,s)R(s) &=& \displaystyle \underbrace{\int_{s}^{t}X(t,\tau)P^{A}(\tau)C(\tau)P^{B}(\tau)Y(\tau,s)d\tau}_{:=D_{1}}\\
&-&\displaystyle\underbrace{\int_{t}^{+\infty}X(t,\tau)Q^{A}(\tau)C(\tau)P^{B}(\tau)Y(\tau,s)d\tau}_{:=D_{2}}\\
&-&\displaystyle\underbrace{\int_{0}^{s}X(t,\tau)P^{A}(\tau)C(\tau)Q^{B}(\tau)Y(\tau,s)d\tau}_{:=D_{3}},
\end{array}
\end{displaymath}
where $I_n-P^{A}(\tau)=Q^{A}(\tau)$ and $I_{m}-P^{B}(\tau)=Q^{B}(\tau)$.

By using again Lemma \ref{LTNU} followed by the Uniformization Lemma and recalling that $s\leq t$, we have that:
\begin{equation}
\label{D1}
\begin{array}{rcl}
||D_{1}||
&\leq& \displaystyle Le^{\theta s}||C||_{\tau,\infty}\int_{s}^{t}||X(t,\tau)P^{A}(\tau)||_{\tau,t}\,||P^{B}(\tau)Y(\tau,s)||_{s,\tau}\,d\tau\\\\
&\leq & \displaystyle Le^{\theta s}||C||_{\tau,\infty}\kappa\widetilde{\kappa}\int_{s}^{t}e^{-\alpha(t-\tau)}e^{-\widetilde{\alpha}(\tau-s)}\,d\tau,
\end{array}
\end{equation}
\begin{equation}
\label{D2}
\begin{array}{rcl}
||D_{2}||
&\leq& \displaystyle
Le^{\theta s}||C||_{\tau, \infty}\int_{t}^{\infty}||X(t,\tau)Q^{A}(\tau)||_{t,\tau}\,||P^{B}(\tau)Y(\tau,s)||_{s,\tau}\,d\tau\\\\
&\leq & \displaystyle
Le^{\theta s}||C||_{\tau,\infty}\kappa\widetilde{\kappa}\int_{t}^{\infty}e^{-\alpha(\tau-t)}e^{-\widetilde{\alpha}(\tau-s)}\,d\tau,
\end{array}
\end{equation}
\begin{equation}
\label{D3}
\begin{array}{rcl}
||D_{3}||&\leq& \displaystyle
Le^{\theta s}||C||_{\tau,\infty}\int_{0}^{s}||X(t,\tau)P^{A}(\tau)||_{\tau,t}\,||Q^{B}(\tau)Y(\tau,s)||_{s,\tau}\,ds\\\\
&\leq & \displaystyle Le^{\theta s}||C||_{\tau,\infty}\kappa\widetilde{\kappa}\int_{0}^{s}e^{-\alpha(t-\tau)}e^{-\widetilde{\alpha}(s-\tau)}\,d\tau.
\end{array}
\end{equation}

As a consequence of the estimates (\ref{D1}), (\ref{D2}), (\ref{D3})  and defining $\kappa_{I}=\max\left \{ \kappa,\widetilde{\kappa}\right \}$, we have:
$$
\begin{array}{rcl}
    &&||W(t,s)P^{B}(s)+X(t,s)R(s)|| \leq \displaystyle Le^{\theta s}||C||_{\tau,\infty}\,\kappa_{I}^{2}\int_{s}^{t}e^{-\alpha(t-\tau)}e^{-\widetilde{\alpha}(\tau-s)}d\tau \\\\
     &+& \displaystyle Le^{\theta s}||C||_{\tau,\infty}\,\kappa_{I}^{2}\int_{t}^{+\infty}e^{-\alpha(\tau-t)}e^{-\widetilde{\alpha}(\tau-s)}d\tau
     + \displaystyle Le^{\theta s}||C||_{\tau,\infty}\,\kappa_{I}^{2}\int_{0}^{s}e^{-\alpha(t-\tau)}e^{-\widetilde{\alpha}(s-\tau)}d\tau\\\\

     &\leq & \displaystyle Le^{\theta s}||C||_{\tau,\infty}\,\kappa_{I}^{2}\cdot\left \{\left (\frac{e^{-\widetilde{\alpha}(t-s)}-e^{-\alpha(t-s)}}{\alpha-\widetilde{\alpha}}\right )
     +\displaystyle \left ( \frac{e^{-\widetilde{\alpha}(t-s)}}{\alpha+\widetilde{\alpha}}\right )
     +\displaystyle  \left ( \frac{e^{-\alpha(t-s)}-e^{-\alpha t-\widetilde{\alpha}s}}{\alpha +\widetilde{\alpha}}\right )\right \}\\\\
     &\leq&\displaystyle Le^{\theta s}||C||_{\tau,\infty}\,\kappa_{I}^{2}\cdot\left [ \frac{1}{|\alpha-\widetilde{\alpha}|}+\frac{1}{\alpha+\widetilde{\alpha}}+\frac{1}{\alpha+\widetilde{\alpha}}\right ]e^{-\alpha_1(t-s)} \\\\
     &\leq &\displaystyle K_1e^{-\alpha_1(t-s)+\theta s},
\end{array}
$$
and if $\alpha \neq \tilde{\alpha}$, then (\ref{DENUa1}) is verified with $\alpha_1=\min\{\alpha,\widetilde{\alpha}\}>\theta$ and 
$$K_1=\max\left \{1,\;\displaystyle L ||C||_{\tau,\infty}\,\kappa_{I}^{2}\cdot\left [ \frac{1}{|\alpha-\widetilde{\alpha}|}+\frac{1}{\alpha+\widetilde{\alpha}}+\frac{1}{\alpha+\widetilde{\alpha}}\right ]\right \}.$$

Note that, similarly as done in \cite{Batelli}, if $\alpha=\widetilde{\alpha}$, we can see that only the first term in the above brackets must be replaced by a new estimation of (\ref{D1}):

$$
\begin{array}{rcl}
||D_{1}||
&\leq & \displaystyle Le^{\theta s}||C||_{\tau,\infty}\kappa\widetilde{\kappa}\int_{s}^{t}e^{-\alpha(t-\tau)}e^{-\alpha(\tau-s)}\,d\tau \\\\
&\leq & \displaystyle Le^{\theta s}||C||_{\tau,\infty}\kappa^{2}_{I}\int_{s}^{t}e^{-\alpha(t-s)}\,d\tau\\\\
&\leq& \displaystyle Le^{\theta s}||C||_{\tau,\infty}\kappa^{2}_{I}(t-s)e^{-\alpha(t-s)},
\end{array}
$$
and since the estimation $(t-s)e^{-\gamma(t-s)}\leq\frac{1}{\gamma e}$, for a positive $\gamma$, if we have that $\gamma<\alpha$ and $\theta<\alpha-\gamma$, then the previous inequality becomes
$$
\begin{array}{rcl}
||D_{1}||
&\leq& \displaystyle Le^{\theta s}||C||_{\tau,\infty}\kappa^{2}_{I}(t-s)e^{-\alpha(t-s)}\leq Le^{\theta s}||C||_{\tau,\infty}\kappa^{2}_{I}\frac{1}{\gamma e}e^{-(\alpha-\gamma)(t-s)},\\
\end{array}
$$
thus we obtain that
$$
\begin{array}{rcl}
 &&||W(t,s)P^{B}(s)+X(t,s)R(s)||\\\\ 
 &\leq& \displaystyle
 Le^{\theta s}||C||_{\tau,\infty}\kappa^{2}_{I}\frac{1}{\gamma e}e^{-(\alpha-\gamma)(t-s)}+Le^{\theta s}||C||_{\tau,\infty}\,\kappa_{I}^{2}\frac{1}{\alpha}e^{-\alpha(t-s)}\\\\
  &\leq& \displaystyle
 Le^{\theta s}||C||_{\tau,\infty}\kappa^{2}_{I}\frac{1}{\gamma }e^{-(\alpha-\gamma)(t-s)}+Le^{\theta s}||C||_{\tau,\infty}\,\kappa_{I}^{2}\frac{1}{\gamma}e^{-(\alpha-\gamma)(t-s)}\\\\
 &\leq&K_2e^{-\alpha_{2}(t-s)+\theta s},
\end{array}
$$
where $K_2=\max\left \{ 1,\; 2L||C||_{\tau,\infty}\kappa^{2}_{I}\frac{1}{\gamma }\right \}$ and $\alpha_2=\alpha-\gamma>\theta$.

Furthermore, if we define $K_3:=\max\{K_1,K_2\}$, $\alpha _3:=\min\{\alpha_{1},\alpha_{2}\}$ and $\varepsilon_{3}:=\theta$, we can conclude the estimate (\ref{DENUa1}).
\end{proof}

\begin{lemma}
\label{DEP}
The evolution operator of (\ref{STB}) and the projector $P(t)$ previously defined verify:
$$
\begin{array}{rcl}
\left \|T(t,s)P(s)\right \|&\leq&\bar{K}e^{-\bar{\alpha}(t-s)+\bar{\varepsilon}s} \quad \textnormal{for any $t\geq s\geq 0$},
\end{array}
$$
where $\bar{K}\geq1$, $\bar{\alpha}>0$, $\bar{\varepsilon}\geq 0$, with $\bar{\alpha}>\bar{\varepsilon}$.
\end{lemma}

\begin{proof}
Let us consider $(\xi_1,\xi_2)\in\R^{n}\times\R^{m}\setminus\{(0,0)\}$, then we have that
$$
\begin{array}{rcl}
\left |T(t,s)P(s)\begin{pmatrix}
\xi_1 \\
\xi_2
\end{pmatrix}\right |&=&\left |\left[\begin{array}{cc}
X(t,s)P^{A}(s) & W(t,s)P^{B}(s)+X(t,s)R(s) \\
0 & Y(t,s)P^{B}(s)
\end{array}\right]\begin{pmatrix}
\xi_1 \\
\xi_2
\end{pmatrix}\right |,\\\\
&=&\left |\left[\begin{array}{cc}
X(t,s)P^{A}(s)\xi_1 + (W(t,s)P^{B}(s)+X(t,s)R(s))\xi_2 \\
 Y(t,s)P^{B}(s)\xi_2
\end{array}\right] \right |\\\\
&\leq&||X(t,s)P^{A}(s) ||\cdot |\xi_1|+\\\\
&&||W(t,s)P^{B}(s)+X(t,s)R(s)||\cdot|\xi_2|+|| Y(t,s)P^{B}(s)||\cdot|\xi_2|
\end{array}
$$
and due to the estimates (\ref{DENUa1}) for the second summand deduced in the above Lemma, the fact that $|\xi_i|\leq |(\xi_1,\xi_2)|$ and the estimations (\ref{DD}) and (\ref{DDD}), the Lemma follows easily.
\end{proof}

The following two lemmas emulate the previous results considering the complementary projectors and $t\leq s$.
Allowing us to end the treatment and study of the dichotomy properties.

\begin{lemma} 
There exist a constant $K_{3}\geq1$, $\alpha_{3}>0$ and $\varepsilon_{3}\geq0$, where $\varepsilon_3<\alpha_3$ such that
\begin{equation}
\label{DENUa2}
||(I_n-P^{A}(t))W(t,s)-R(t)Y(t,s)||\leq K_{3}e^{-\alpha_3(s-t)+\varepsilon_{3}s} \quad \textnormal{for 
$s\geq t\geq 0$}.
\end{equation}
\end{lemma}
\begin{proof}
The proof is a charbon copy of the proof of the Lemma \ref{Lema-W} and is left for the reader.
\end{proof}

\begin{lemma}
\label{DEIP}
The evolution operator of (\ref{STB}) and the projector $P(t)$ previously defined verify:
$$
\begin{array}{rcl}
\left \|T(t,s)[I-P(s)]\right \|&\leq&\bar{K}e^{-\bar{\alpha}(s-t)+\bar{\varepsilon}s} \quad \textnormal{for any $s\geq t\geq 0$},
\end{array}
$$
where $\bar{K}\geq1$, $\bar{\alpha}>0$, $\bar{\varepsilon}\geq 0$, with $\bar{\alpha}>\bar{\varepsilon}$.
\end{lemma}

\begin{proof}
Let us consider $(\xi_1,\xi_2)\in\R^{n}\times\R^{m}\setminus\{(0,0)\}$, then we have that
$$
\begin{array}{rcl}
\left |T(t,s)(I-P(s))\begin{pmatrix}
\xi_1 \\
\xi_2
\end{pmatrix}\right |&=&\left |(I-P(t))T(t,s)\begin{pmatrix}
\xi_1 \\
\xi_2
\end{pmatrix}\right |\\\\
&=&\left |\left[\begin{array}{cc}
Q^{A}(t)X(t,s) & Q^{A}(t)W(t,s)-R(t)Y(t,s) \\
0 & Q^{B}(t)Y(t,s)
\end{array}\right]\begin{pmatrix}
\xi_1 \\
\xi_2
\end{pmatrix}\right |,\\\\
&=&\left |\left[\begin{array}{cc}
Q^{A}(t)X(t,s)\xi_1 + (Q^{A}(t)W(t,s)-R(t)Y(t,s))\xi_2 \\
 Q^{B}(t)Y(t,s)\xi_2
\end{array}\right] \right |\\\\
&\leq&||Q^{A}(t)X(t,s) ||\cdot |\xi_1|+\\\\
&&||Q^{A}(t)W(t,s)-R(t)Y(t,s)||\cdot|\xi_2|+|| Q^{B}(t)Y(t,s)||\cdot|\xi_2|
\end{array}
$$
and the Lemma is a consequence of (\ref{DENUa2}), which  estimates the second summand, combined with the fact that $|\xi_i|\leq |(\xi_1,\xi_2)|$.
\end{proof}

The next result shows that the nondiagonal
submatrix $W(t,s)$ of $T(t,s)$ has a property reminiscent to the half nonuniform bounded growth.

\begin{lemma}
There exist a constant $M_{3}\geq1$, $\omega_{3}>0$ and $\theta\geq0$ such that the operator $W(t,s)$, defined in (\ref{W}), verifies
$$||W(t,s)||\leq M_{3}e^{\omega_{3} (t-s)+\theta s},\quad t\geq s.
$$
\end{lemma}

\begin{proof} 
Let us recall that the systems $\dot{x}=A(t)x$ and $\dot{y}=B(t)y$ have a half nonuniform bounded growth on $\R_{0}^{+}$ described in (\ref{HNBG}). By using Lemma \ref{LTNU}, where $Z(\tau)=I_{n}$ and $V(\tau)=I_{m}$, combined with
(\ref{LUPBG}) which arises from the Uniformization Lemma, we can see that when $t\geq s\geq 0$:
$$
\begin{array}{rcl}
||W(t,s)||&\leq&\displaystyle Le^{\theta s}||C||_{\tau,\infty}\int_{s}^{t}||X(t,\tau)||_{\tau,t} ||Y(\tau,s)||_{s,\tau}\; d\tau\\\\
&\leq&Le^{\theta s}||C||_{\tau,\infty}\mu\tilde{\mu} \displaystyle\int_{s}^{t}e^{\omega (t-\tau)}e^{\tilde{\omega}(\tau-s)}\; d\tau.
\end{array}
$$

Here we have two cases. The first one is when $\omega\neq\tilde{\omega}$, then if $\mu_{I}=\max\{\mu,\tilde{\mu}\}$:
$$
\begin{array}{rcl}
||W(t,s)||&\leq&Le^{\theta s}||C||_{\tau,\infty}\mu_{I}^{2} e^{\omega t-\tilde{\omega} s}\displaystyle\int_{s}^{t}e^{(\omega-\tilde{\omega})\tau}\; d\tau\\\\

&=&\displaystyle Le^{\theta s}||C||_{\tau,\infty}\mu_{I}^{2} \frac{1}{|\tilde{\omega}-\omega|}\left [e^{\tilde{\omega} (t-s)}-e^{\omega(t-s)}\right ]\\\\
&\leq&\displaystyle M_1 e^{\omega_1(t-s)+\theta s},
\end{array}
$$
where 
$$M_1=\max\left \{ 1, \displaystyle L||C||_{\tau,\infty}\mu_{I}^{2} \frac{2}{|\tilde{\omega}-\omega|}\right \}\quad \textnormal{and}\quad \omega_{1}=\max\{\omega,\tilde{\omega}\}.$$

The second case is when $\omega=\tilde{\omega}$, then
$$
\begin{array}{rcl}
||W(t,s)||&\leq&Le^{\theta s}||C||_{\tau,\infty}\mu_{I}^{2}\displaystyle\int_{s}^{t}e^{\omega(t-s)}\; d\tau\\\\
&\leq&\displaystyle Le^{\theta s}||C||_{\tau,\infty}\mu_{I}^{2} e^{\omega(t-s)}e^{(t-s)}\\\\
&\leq&\displaystyle M_2 e^{\omega_{2}(t-s)+\theta s},
\end{array}
$$
where
$$M_2=\max \left \{ 1, L||C||_{\tau,\infty}\mu_{I}^{2}  \right \}\quad \textnormal{and}\quad \omega_{2}=\omega+1.$$

Based on the two cases analyzed, we can conclude that 
\begin{equation}
\label{HNBGW}
||W(t,s)||\leq Me^{\omega_{3} (t-s)+\theta s},\quad t\geq s,
\end{equation}
where 
$$M_3=\max \left \{M_1, M_2\right \}\quad \textnormal{and}\quad \omega_{3}=\max\{\omega_1,\omega_2\}.$$
\end{proof}

The last result shows that the evolution operator associated to the upper triangular system (\ref{lin}) has the property of
half $(\overline{M}e^{\overline{\theta}s},\overline{\omega})$--nonuniform bounded growth.

\begin{lemma}
\label{UL1}
The evolution operator of (\ref{STB}) verify:
$$
\begin{array}{rcl}
\left \|T(t,s)\right \|&\leq&\bar{M}e^{\bar{\omega}(t-s)+\bar{\theta}s} \quad \textnormal{for any $t\geq s\geq 0$},
\end{array}
$$
where $\bar{M}\geq1$, $\bar{\omega}>0$, $\bar{\theta}\geq 0$.
\end{lemma}

\begin{proof}
If we consider $(\xi_1,\xi_2)\in\R^{n}\times\R^{m}\setminus\{(0,0)\}$, then we have that
$$
\begin{array}{rcl}
\left |T(t,s)\begin{pmatrix}
\xi_1 \\
\xi_2
\end{pmatrix}\right |&=&\left |\left[\begin{array}{cc}
X(t,s) & W(t,s) \\
0 & Y(t,s)
\end{array}\right]\begin{pmatrix}
\xi_1 \\
\xi_2
\end{pmatrix}\right |,\\\\
&=&\left |\left[\begin{array}{cc}
X(t,s)\xi_1 + W(t,s)\xi_2 \\
 Y(t,s)\xi_2
\end{array}\right] \right |\\\\
&\leq&||X(t,s) ||\cdot |\xi_1|+\\\\
&&||W(t,s)||\cdot|\xi_2|+|| Y(t,s)||\cdot|\xi_2|
\end{array}
$$
and due to the estimation (\ref{HNBGW}), the fact that $|\xi_i|\leq |(\xi_1,\xi_2)|$ and both estimations in (\ref{HNBG}), we can ensure that for $t\geq s$:
$$||T(t,s)||\leq \bar{M}e^{\bar{\theta} s}e^{\bar{\omega}(t-s)}.$$
\end{proof}

\subsection{End of proof of Theorem \ref{vietnam}}

Firstly, the Lemmas \ref{T-INV}, \ref{DEP} and \ref{DEIP} imply that the triangular system (\ref{STB}) has a $(\bar{K}e^{\bar{\varepsilon}s},\bar{\alpha})$--nonuniform exponential dichotomy in $\mathbb{R}_{0}^{+}$. 

Secondly, the Lemma \ref{UL1} says that the system (\ref{STB}) has the property of half $(\bar{M}e^{\bar{\theta}s},\bar{ \omega})$--nonuniform bounded growth and the Theorem follows.

\begin{remark}
A meticulous reading of this Appendix shows that the property of half nonuniform bounded growth 
is fundamental in several steps of the proof: 
\begin{itemize}
\item[a)] Is a necessary condition in order to use the Uniformization Lemma, which ensures the existence of a continuous family norms $\{|\cdot|_{t}\}$ verifying the inequality $|x|_{t}\leq L_{2}(t)|x|$. The half nonuniform bounded growth property is a required tool to obtain explicit estimations for $L_{2}(\cdot)$ and (\ref{cotas-beta}).
\item[b)] The constants $\ell$ and $\tilde{\ell}$ are necessary to deduce (\ref{cotas-beta}) and (\ref{cota-L}), these identities are immersed in Lemma \ref{LTNU}, which is the main key to deduce several estimations around the proof.
\item[c)] The previous facts, also shows that the boundedness properties of $||C(\tau)||_{\tau,\tau}$ involves estimations based in the half nonuniform bounded growth property.
\end{itemize}

We point out that in \cite{Tien}, the property of half nonuniform bounded growth is not considered neither in the statement of Uniformization Lemma (Lemma 2.2 in \cite{Tien}) nor in the statement of Theorem 2.3.

\end{remark}



\end{document}